\documentclass[reqno]{amsart}
\usepackage{amssymb,amsmath,amsfonts,amsthm,graphicx,ifthen,mathrsfs}
\usepackage{amsaddr} % author affiliations on first page
\usepackage[mathcal]{euscript}
\usepackage{hyperref} % for urls in bib
\usepackage{graphicx} % Required for inserting images
\usepackage[dvipsnames]{xcolor} %for textcolors
\usepackage{tikz}
\usetikzlibrary{arrows.meta}
\usepackage{pgfplots}
\pgfplotsset{compat=1.15}
\usepackage{mathrsfs}
\usetikzlibrary{arrows}
\usetikzlibrary{positioning}
\usepackage{xspace}  % for textsc in title
\usepackage{textcase} % for title
\usepackage{stmaryrd} % for up arrows
\usepackage{tabularx} % for tables
\usepackage{adjustbox} 

\theoremstyle{plain}
\newtheorem{Theorem}{Theorem}   [section]

\theoremstyle{definition}
\newtheorem{Definition}[Theorem] {Definition}

\usepackage[capitalise, noabbrev]{cleveref}
	\crefname{Construction}{Construction}{Constructions}
	\Crefname{Construction}{Construction}{Constructions}

\newcommand{\craig}[1]{\textcolor{orange}{#1}}

\newcommand{\setC}{\ensuremath{\mathcal{C}}}

\newcommand{\corange}[1]{\textcolor{orange}{#1}}

\newcommand{\ccyan}[1]{\textcolor{cyan}{#1}}

\newcommand{\lequiv}{\ensuremath{\sim}}

\begin{document}

% Or "Measured play:..."
\title{Mind the gap: A real-valued distance on combinatorial games}
\author{K. Burke$^1$, M. Fisher$^2$, C. Tennenhouse$^3$}
\address{$^1$Dept. of Comp. Sci., FL Southern College, Lakeland, FL USA 33801}
\address{$^2$Dept. of Mathematics, West Chester Univ., West Chester, PA, USA 19383}
\address{$^3$School of Math. \& Data Sci., Univ. of New England, Biddeford, ME, USA 04005}
\email{$^1$kburke@flsouthern.edu}
\email{$^2$mfisher@wcupa.edu}
\email{$^3$ctennenhouse@une.edu}

\maketitle

% abstract here
\begin{abstract}
We define a real-valued distance metric $wd$ on the space $\mathcal{C}$ of short combinatorial games in canonical form. We demonstrate the existence of Cauchy sequences informed by sidling sequences, find limit points, and investigate the closure $\overline{\mathcal{C}}$, which is shown to partition the set of loopy games in a non-trivial way. Stoppers, enders, and non-stopper-sided loopy games are explored, as well as the topological properties of $(\mathcal{C},wd)$.
\end{abstract}

\section{Introduction}\label{sec:intro}

% short intro to combinatorial games, reference to LiP and/or wwfymp

A \emph{combinatorial game} is a two-player turn-based strategy game with perfect information and no chance elements. Under the \emph{normal play} convention the first player unable to move on their turn is the loser. Players are traditionally referred to as Left (with moves and pieces colored blue) and Right (with moves and pieces colored red). By convention we denote a game by $G = \{G^L|G^R\}$, where $G^L$ is the set of games Left can move to on their turn (called Left's \emph{options}), and $G^R$ the set of options available to Right. For a full introduction to Combinatorial Game Theory see \cite{WinningWays:2001} and \cite{LessonsInPlay:2007}. Throughout this manuscript we will use the notation $\{x|y||z|w\}$ (and its recursive analog) to denote $\{\{x|y\}|\{z|w\}\}$. We make use of a bespoke Python script for generating tikz diagrams to visualize game graphs, which is dependent on the excellent CGSuite software package \cite{SiegelCGSuite:2003}. In addition, many of our calculations could not have been completed without CGSuite, in which we both used default functions and wrote a number of scripts for this express purpose. 

% short comb games and canonical form
%	"collapsing" in canonical form
%	domination and reversibility
\begin{Definition}
A \emph{follower} in a combinatorial game is a game reachable from that game during play. A combinatorial game is \emph{short} if no position can be repeated and there are only a finite number of followers.
\end{Definition}

A short combinatorial game is in \emph{canonical form} if it contains no dominated or reversible options. For example, the game $\{*,1|0\}$ simplifies to $\{1|0\}$ since $*$ is dominated by $1$ for Left. So $\{*,1|0\}$ is not in canonical form. For the definition of reversibility we paraphrase \cite{LessonsInPlay:2007}.

\begin{Definition}
A left option $A$ of $G$ is \emph{reversible} if $A$ has a right option $A^R$ such that $A^R \leq G$.
\end{Definition}

If $G = \{A, B, C, \ldots |H, I, J, \ldots\}$ and $A^R\leq G$ with $\{W, X, Y, \ldots\}$ the Left options of $A^R$, then we can rewrite $G$ as $\{W, X, Y, \ldots ,B, C, \ldots|H, I, J, \ldots\}$.

We let \setC\ denote the set of all short combinatorial games in canonical form.

% short games as acyclic colored multi-digraphs
A short game can be visualized as a rooted tree with edges colored blue and red. However, since every short game $G$ has a single source (the game itself) and a single sink (the game $0$) we can also represent any short game $G$ by a bicolored directed acyclic graph, or \emph{DAG}, which can be generated from a tree by identifying nodes that represent the same game. We will use $D(G)$ to represent the DAG of short game $G$, or simply $G$ where the context is clear. See Fig. \ref{fig:fractions} for an example.

As another example consider the game $G = \{*,1|0\}$ from above, which we saw is not in canonical form. Because $G$ is equivalent to the game $\{1|0\}\in \setC$, we say that $D(G)$ \emph{collapses} to $D(\{1|0\})$. If $G = \{x|y\}$ with $x<y$ both numbers then $G$ is commonly simplified to what is referred as \emph{the simplest number} between $x$ and $y$. That is, $\{x|y\}$ is the smallest dyadic rational (i.e. number of the form $\frac{a}{2^b}$ in lowest terms) between $x$ and $y$ among all such dyadic rationals with the smallest denominator. 

% followers, birthday
% distance metric definition, metric spaces
% edit distance is already a distance metric, but we want to both capture natural alterations to positions based on gameplay, and capture convergence
% other work on analysis on surreal numbers (Simon and his student)

% loopy and infinite games 
Not all combinatorial games are short. Some involve the possibility of infinite play.

\begin{Definition}
A game is \emph{loopy} if it has a follower that can be returned to at some point during play. 
\end{Definition}

If a game is loopy then it cannot be represented by a finite DAG, as a finite representative graph requires cycles. There is also no currently agreed upon convention for canonical forms of loopy games (though there has been work on expressing a loopy game in its \emph{simplest form} \cite{ConwayLoopy:1978,NewResultsMSRI:Siegel:2009}). However, if we allow for infinite graphs then we can realize loopy games without the need for cycles. Consider $\textsc{on} = \{\textsc{on}|\}$, the game where Left can move back to $\textsc{on}$ and Right has no options. Traditionally \textsc{on} is visualized as a single vertex with a single blue loop, but we can also draw it as an infinite directed path of blue edges. 

\iffalse %%%%%%%%%%%%%%%%%%%%%%%
\craig{We can't actually formalize this. As we see with dud, there can be multiple infinite digraphs that play the same as a loopy game. Instead, let's address this later when we define wd by saying any game in \setC\ that plays just like a loopy game up to move n, for any n, can represent that loopy game, and define wd to be the inf over all representative loop-free infinite graphs that represent the loopy game.
}
Let's formalize this construction. If $G$ is a loopy game then we define $D(G)$, an infinite DAG, in the following way. Begin by representing the game $G$ with a bicolored directed graph with cycles. Say $e$ is an edge from $y$ to a game $x$ that appears on a path from $G$ to $y$, forming a cycle. We replace $e$ with an edge of the same color as $e$ from $y$ to a new vertex $x'$ such that the subgraph rooted at $x$ is isomorphic to the subgraph rooted at $x'$. We continue this process until no cycles remain, potentially replacing infinitely many copies of a cycle representing a single play at once. The graph that results is $D(G)$.
\fi %%%%%%%%%%%%%%%%%%%%%%%

Note that not all games that are not short are loopy. We also have loop-free \emph{infinite games}, games with an infinite number of followers. Consider, for example, the game $\omega$, realized in game notation as $\{0,1,2,\ldots |\}$. Left has infinitely many options but there is no repetition possible, and the game will always end in a finite number of moves.

%	plum trees
A loopy game in which all cycles have length $1$ is called a \emph{plum tree}, of which \textsc{on} is a named example, along with $\textsc{dud}=\{\textsc{dud}|\textsc{dud}\}$ (which stands for \emph{deathless universal draw}), $\textsc{over} = \{0|\textsc{over}\}, \textsc{upon} = \{\textsc{upon}|*\}$, and many others. Of note is that the sum of $\textsc{on}$ and $\textsc{off}$ (another name for $-\textsc{on}$) is not $0$ but \textsc{dud}, since neither player has a move other than to return to the starting position. 

% define stoppers (games that end under alternating play)
Continuing our definitions around non-short games, it's important to differentiate between games that \emph{can} stop and those that \emph{will} stop. We now define two new terms simultaneously.

% define enders (no infinite play at all from any vertex)
\begin{Definition}
A game that, when played in isolation in alternating play, cannot continue forever without stopping is a \emph{stopper}. A game in which it is impossible to have an infinite series of moves, even under non-alternating play, is an \emph{ender}. 
\end{Definition}

So \textsc{on} and \textsc{over} are stoppers since alternating play results in play stopping after at most a couple moves, but neither is an ender since one player could play continuously on the same game. However, \textsc{dud} is neither a stopper nor an ender. 

Games that are not short are often analyzed by means of their \emph{sides}, which are game values that assist in the analysis of non-stoppers when the one or both players are considered to win (or lose) if the game continues forever. Some games with this property include \textsc{Bagh Chal}, \textsc{Fox \& Geese}, and \textsc{Hare \& Hounds}. We will not address the theory of sides in this manuscript except as a motivation for \emph{sidling}. 

% define sidling
\begin{Definition}
\emph{Sidling} is the process of generating a sequence of game values (a \emph{sidling sequence}), possibly infinite, that approaches the sides of a loopy game $\gamma$ by beginning with a value known to be greater than (less than) $\gamma$ and successively generating a smaller (larger) term in the sequence.
\end{Definition}

%	purpose is to find the sides of a game, which tell us how to treat the game under different winning conditions
%	example from hare and hounds
%	cite lots of Aaron's work
%	example of (1/2)^n and (*+(^+*)n) both approaching over, from Siegel

We demonstrate by reproducing an example from page 291 in \cite{SiegelCGT:2013}. Consider the game $\textsc{over}$ as defined above. We begin with the fact that $\textsc{over}\leq \textsc{on}$, and since  $\textsc{over} = \{0|\textsc{over}\}$ we have
\[
\hspace*{-3in}
\begin{aligned}
\textsc{over} &\leq  \{0|\textsc{on}\} &= 1 \\
\textsc{over} &\leq  \{0|1\} &= \frac{1}{2} \\
\textsc{over} &\leq  \{0|\frac{1}{2}\} &= \frac{1}{4} \\
\textsc{over} &\leq  \{0|\frac{1}{4}\} &= \frac{1}{8} \\
\vdots && \\
\end{aligned}
\]
This yields the sequence $(\frac{1}{2^n})$. Each term of this sequence remains greater than $\textsc{over}$. Similarly, we can sidle from below (called the \emph{\textsc{off}-side}). Note that $\textsc{over} \geq \textsc{off}$, but it is also greater than $0$, which is a simpler value to start with.
\[
\hspace*{-3in}
\begin{aligned}
\textsc{over} &\geq  \{0|0\} &=& * \\
\textsc{over} &\geq  \{0|*\} &=& \uparrow \\
\textsc{over} &\geq  \{0|\uparrow\} &=& \Uparrow\!* \\
\textsc{over} &\geq  \{0|\Uparrow\!*\} &=& \uparrow\uparrow\uparrow \\
\vdots && \\
\end{aligned}
\]

which yields the sidling sequence $(*+(\uparrow\!*)\times n)$. Figures  \ref{fig:fractions} and \ref{fig:ups} show a few terms of each of these sequences. 

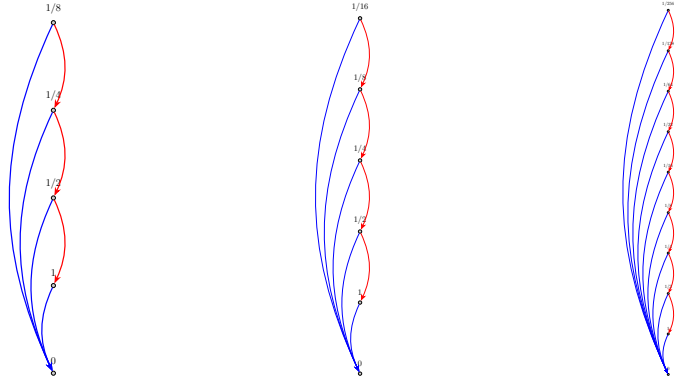
\begin{figure}[ht]
\centering
\begin{minipage}{0.32\textwidth}
  \centering
  \resizebox{!}{5cm}{\begin{tikzpicture}[>=Stealth, auto, node distance=2cm, line width=1.2pt]
\tikzset{
  L_edge/.style={->, very thick, blue},
  R_edge/.style={->, very thick, red},
  loop_large/.style={loop, in=135, out=45, looseness=20, distance=50pt}
}
\node[circle, inner sep=1pt, minimum size=4pt, draw=black, line width=0.5pt, fill=gray!40, label={[font=\bfseries, label distance=0.1cm]above:$1/8$}] (N1) at (0.0, 12.0) {};
\node[circle, inner sep=1pt, minimum size=4pt, draw=black, line width=0.5pt, fill=gray!40, label={[font=\bfseries, label distance=0.1cm]above:$0$}] (N2) at (0.0, 0.0) {};
\node[circle, inner sep=1pt, minimum size=4pt, draw=black, line width=0.5pt, fill=gray!40, label={[font=\bfseries, label distance=0.1cm]above:$1/4$}] (N3) at (0.0, 9.0) {};
\node[circle, inner sep=1pt, minimum size=4pt, draw=black, line width=0.5pt, fill=gray!40, label={[font=\bfseries, label distance=0.1cm]above:$1/2$}] (N4) at (0.0, 6.0) {};
\node[circle, inner sep=1pt, minimum size=4pt, draw=black, line width=0.5pt, fill=gray!40, label={[font=\bfseries, label distance=0.1cm]above:$1$}] (N5) at (0.0, 3.0) {};
\draw[L_edge] (N1) to[bend right=25] (N2);
\draw[R_edge] (N1) to[bend left=25] (N3);
\draw[L_edge] (N3) to[bend right=25] (N2);
\draw[R_edge] (N3) to[bend left=25] (N4);
\draw[L_edge] (N4) to[bend right=25] (N2);
\draw[R_edge] (N4) to[bend left=25] (N5);
\draw[L_edge] (N5) to[bend right=25] (N2);
\end{tikzpicture}}
\end{minipage}%
\begin{minipage}{0.32\textwidth}
  \centering
  \resizebox{!}{5cm}{\begin{tikzpicture}[>=Stealth, auto, node distance=2cm, line width=1.2pt]
\tikzset{
  L_edge/.style={->, very thick, blue},
  R_edge/.style={->, very thick, red},
  loop_large/.style={loop, in=135, out=45, looseness=20, distance=50pt}
}
\node[circle, inner sep=1pt, minimum size=4pt, draw=black, line width=0.5pt, fill=gray!40, label={[font=\bfseries, label distance=0.1cm]above:$1/16$}] (N1) at (0.0, 15.0) {};
\node[circle, inner sep=1pt, minimum size=4pt, draw=black, line width=0.5pt, fill=gray!40, label={[font=\bfseries, label distance=0.1cm]above:$0$}] (N2) at (0.0, 0.0) {};
\node[circle, inner sep=1pt, minimum size=4pt, draw=black, line width=0.5pt, fill=gray!40, label={[font=\bfseries, label distance=0.1cm]above:$1/8$}] (N3) at (0.0, 12.0) {};
\node[circle, inner sep=1pt, minimum size=4pt, draw=black, line width=0.5pt, fill=gray!40, label={[font=\bfseries, label distance=0.1cm]above:$1/4$}] (N4) at (0.0, 9.0) {};
\node[circle, inner sep=1pt, minimum size=4pt, draw=black, line width=0.5pt, fill=gray!40, label={[font=\bfseries, label distance=0.1cm]above:$1/2$}] (N5) at (0.0, 6.0) {};
\node[circle, inner sep=1pt, minimum size=4pt, draw=black, line width=0.5pt, fill=gray!40, label={[font=\bfseries, label distance=0.1cm]above:$1$}] (N6) at (0.0, 3.0) {};
\draw[L_edge] (N1) to[bend right=25] (N2);
\draw[R_edge] (N1) to[bend left=25] (N3);
\draw[L_edge] (N3) to[bend right=25] (N2);
\draw[R_edge] (N3) to[bend left=25] (N4);
\draw[L_edge] (N4) to[bend right=25] (N2);
\draw[R_edge] (N4) to[bend left=25] (N5);
\draw[L_edge] (N5) to[bend right=25] (N2);
\draw[R_edge] (N5) to[bend left=25] (N6);
\draw[L_edge] (N6) to[bend right=25] (N2);
\end{tikzpicture}}
\end{minipage}%
\begin{minipage}{0.32\textwidth}
  \centering
  \resizebox{!}{5cm}{\begin{tikzpicture}[>=Stealth, auto, node distance=2cm, line width=1.2pt]
\tikzset{
  L_edge/.style={->, very thick, blue},
  R_edge/.style={->, very thick, red},
  loop_large/.style={loop, in=135, out=45, looseness=20, distance=50pt}
}
\node[circle, inner sep=1pt, minimum size=4pt, draw=black, line width=0.5pt, fill=gray!40, label={[font=\bfseries, label distance=0.1cm]above:$1/256$}] (N1) at (0.0, 27.0) {};
\node[circle, inner sep=1pt, minimum size=4pt, draw=black, line width=0.5pt, fill=gray!40, label={[font=\bfseries, label distance=0.1cm]above:$0$}] (N2) at (0.0, 0.0) {};
\node[circle, inner sep=1pt, minimum size=4pt, draw=black, line width=0.5pt, fill=gray!40, label={[font=\bfseries, label distance=0.1cm]above:$1/128$}] (N3) at (0.0, 24.0) {};
\node[circle, inner sep=1pt, minimum size=4pt, draw=black, line width=0.5pt, fill=gray!40, label={[font=\bfseries, label distance=0.1cm]above:$1/64$}] (N4) at (0.0, 21.0) {};
\node[circle, inner sep=1pt, minimum size=4pt, draw=black, line width=0.5pt, fill=gray!40, label={[font=\bfseries, label distance=0.1cm]above:$1/32$}] (N5) at (0.0, 18.0) {};
\node[circle, inner sep=1pt, minimum size=4pt, draw=black, line width=0.5pt, fill=gray!40, label={[font=\bfseries, label distance=0.1cm]above:$1/16$}] (N6) at (0.0, 15.0) {};
\node[circle, inner sep=1pt, minimum size=4pt, draw=black, line width=0.5pt, fill=gray!40, label={[font=\bfseries, label distance=0.1cm]above:$1/8$}] (N7) at (0.0, 12.0) {};
\node[circle, inner sep=1pt, minimum size=4pt, draw=black, line width=0.5pt, fill=gray!40, label={[font=\bfseries, label distance=0.1cm]above:$1/4$}] (N8) at (0.0, 9.0) {};
\node[circle, inner sep=1pt, minimum size=4pt, draw=black, line width=0.5pt, fill=gray!40, label={[font=\bfseries, label distance=0.1cm]above:$1/2$}] (N9) at (0.0, 6.0) {};
\node[circle, inner sep=1pt, minimum size=4pt, draw=black, line width=0.5pt, fill=gray!40, label={[font=\bfseries, label distance=0.1cm]above:$1$}] (N10) at (0.0, 3.0) {};
\draw[L_edge] (N1) to[bend right=25] (N2);
\draw[R_edge] (N1) to[bend left=25] (N3);
\draw[L_edge] (N3) to[bend right=25] (N2);
\draw[R_edge] (N3) to[bend left=25] (N4);
\draw[L_edge] (N4) to[bend right=25] (N2);
\draw[R_edge] (N4) to[bend left=25] (N5);
\draw[L_edge] (N5) to[bend right=25] (N2);
\draw[R_edge] (N5) to[bend left=25] (N6);
\draw[L_edge] (N6) to[bend right=25] (N2);
\draw[R_edge] (N6) to[bend left=25] (N7);
\draw[L_edge] (N7) to[bend right=25] (N2);
\draw[R_edge] (N7) to[bend left=25] (N8);
\draw[L_edge] (N8) to[bend right=25] (N2);
\draw[R_edge] (N8) to[bend left=25] (N9);
\draw[L_edge] (N9) to[bend right=25] (N2);
\draw[R_edge] (N9) to[bend left=25] (N10);
\draw[L_edge] (N10) to[bend right=25] (N2);
\end{tikzpicture}}
\end{minipage}%
\caption{Digraph representations of the games $\frac{1}{8}$, $\frac{1}{16}$, and $\frac{1}{256}$}
\label{fig:fractions}
\end{figure}

\begin{figure}[ht]
\centering
\begin{minipage}{0.32\textwidth}
  \centering
  \resizebox{!}{5cm}{% Original Game: {{|}|{{|}|{{|}|{|}}}}
\begin{tikzpicture}[>=Stealth, auto, node distance=2cm, line width=1.2pt]
\tikzset{
  L_edge/.style={->, very thick, blue},
  R_edge/.style={->, very thick, red},
  loop_large/.style={loop, in=135, out=45, looseness=20, distance=50pt}
}
\node[circle, inner sep=1pt, minimum size=4pt, draw=black, line width=0.5pt, fill=white] (N1) at (0.0, 9.0) {};
\node[circle, inner sep=1pt, minimum size=4pt, draw=black, line width=0.5pt, fill=white] (N2) at (0.0, 0.0) {};
\node[circle, inner sep=1pt, minimum size=4pt, draw=black, line width=0.5pt, fill=white] (N3) at (0.0, 6.0) {};
\node[circle, inner sep=1pt, minimum size=4pt, draw=black, line width=0.5pt, fill=white] (N4) at (0.0, 3.0) {};
\draw[L_edge] (N1) to[bend right=25] (N2);
\draw[R_edge] (N1) to[bend left=25] (N3);
\draw[L_edge] (N3) to[bend right=25] (N2);
\draw[R_edge] (N3) to[bend left=25] (N4);
\draw[L_edge] (N4) to[bend right=25] (N2);
\draw[R_edge] (N4) to[bend left=25] (N2);
\end{tikzpicture}}
\end{minipage}%
\begin{minipage}{0.32\textwidth}
  \centering
  \resizebox{!}{5cm}{% Original Game: {{|}|{{|}|{{|}|{{|}|{{|}|{|}}}}}}
\begin{tikzpicture}[>=Stealth, auto, node distance=2cm, line width=1.2pt]
\tikzset{
  L_edge/.style={->, very thick, blue},
  R_edge/.style={->, very thick, red},
  loop_large/.style={loop, in=135, out=45, looseness=20, distance=50pt}
}
\node[circle, inner sep=1pt, minimum size=4pt, draw=black, line width=0.5pt, fill=white] (N1) at (0.0, 15.0) {};
\node[circle, inner sep=1pt, minimum size=4pt, draw=black, line width=0.5pt, fill=white] (N2) at (0.0, 0.0) {};
\node[circle, inner sep=1pt, minimum size=4pt, draw=black, line width=0.5pt, fill=white] (N3) at (0.0, 12.0) {};
\node[circle, inner sep=1pt, minimum size=4pt, draw=black, line width=0.5pt, fill=white] (N4) at (0.0, 9.0) {};
\node[circle, inner sep=1pt, minimum size=4pt, draw=black, line width=0.5pt, fill=white] (N5) at (0.0, 6.0) {};
\node[circle, inner sep=1pt, minimum size=4pt, draw=black, line width=0.5pt, fill=white] (N6) at (0.0, 3.0) {};
\draw[L_edge] (N1) to[bend right=25] (N2);
\draw[R_edge] (N1) to[bend left=25] (N3);
\draw[L_edge] (N3) to[bend right=25] (N2);
\draw[R_edge] (N3) to[bend left=25] (N4);
\draw[L_edge] (N4) to[bend right=25] (N2);
\draw[R_edge] (N4) to[bend left=25] (N5);
\draw[L_edge] (N5) to[bend right=25] (N2);
\draw[R_edge] (N5) to[bend left=25] (N6);
\draw[L_edge] (N6) to[bend right=25] (N2);
\draw[R_edge] (N6) to[bend left=25] (N2);
\end{tikzpicture}}
\end{minipage}%
\begin{minipage}{0.32\textwidth}
  \centering
  \resizebox{!}{5cm}{% Original Game: {{|}|{{|}|{{|}|{{|}|{{|}|{{|}|{|}}}}}}}
\begin{tikzpicture}[>=Stealth, auto, node distance=2cm, line width=1.2pt]
\tikzset{
  L_edge/.style={->, very thick, blue},
  R_edge/.style={->, very thick, red},
  loop_large/.style={loop, in=135, out=45, looseness=20, distance=50pt}
}
\node[circle, inner sep=1pt, minimum size=4pt, draw=black, line width=0.5pt, fill=white] (N1) at (0.0, 18.0) {};
\node[circle, inner sep=1pt, minimum size=4pt, draw=black, line width=0.5pt, fill=white] (N2) at (0.0, 0.0) {};
\node[circle, inner sep=1pt, minimum size=4pt, draw=black, line width=0.5pt, fill=white] (N3) at (0.0, 15.0) {};
\node[circle, inner sep=1pt, minimum size=4pt, draw=black, line width=0.5pt, fill=white] (N4) at (0.0, 12.0) {};
\node[circle, inner sep=1pt, minimum size=4pt, draw=black, line width=0.5pt, fill=white] (N5) at (0.0, 9.0) {};
\node[circle, inner sep=1pt, minimum size=4pt, draw=black, line width=0.5pt, fill=white] (N6) at (0.0, 6.0) {};
\node[circle, inner sep=1pt, minimum size=4pt, draw=black, line width=0.5pt, fill=white] (N7) at (0.0, 3.0) {};
\draw[L_edge] (N1) to[bend right=25] (N2);
\draw[R_edge] (N1) to[bend left=25] (N3);
\draw[L_edge] (N3) to[bend right=25] (N2);
\draw[R_edge] (N3) to[bend left=25] (N4);
\draw[L_edge] (N4) to[bend right=25] (N2);
\draw[R_edge] (N4) to[bend left=25] (N5);
\draw[L_edge] (N5) to[bend right=25] (N2);
\draw[R_edge] (N5) to[bend left=25] (N6);
\draw[L_edge] (N6) to[bend right=25] (N2);
\draw[R_edge] (N6) to[bend left=25] (N7);
\draw[L_edge] (N7) to[bend right=25] (N2);
\draw[R_edge] (N7) to[bend left=25] (N2);
\end{tikzpicture}}
\end{minipage}%
\caption{Digraph representations of $\Uparrow\!*$, $\uparrow\!4*$, and $\uparrow\!5$}
\label{fig:ups}
\end{figure}
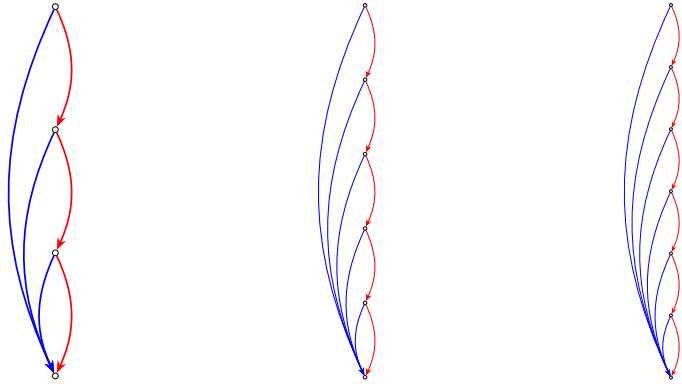

% purpose is two-fold
%	analytic study of the space of short games yielding insight into structure and relationships among objects
%	optimization requires convergence; discuss GA/simulated annealing attempts and future ideas
%	useful to have a non-negative real-valued distance between objects 

Sidling does not always yield an infinite sequence. Sometimes it reaches a side after a finite number of steps, and sometimes it does not yield a useful sequence at all. 

The astute reader will notice that terms of one of our example sidling sequences are nested within the other as subgraphs, and in fact terms of one sequence are very similar topologically to terms of the other. In Section \ref{sec:dwd} we formalize these notions. Our primary goal in this project is the analytic study of the space \setC\ of short games in canonical form, yielding insight into their structure and relationships to each other, and to the loopy games that they approximate. One purpose of this endeavor is to better understand and express what it means for a pair or larger collection of games to be \emph{close}, hopefully allowing for future examination of games as subjects of optimization. For that reason we need to discuss \emph{metric spaces}.

\section{Sequences in metric spaces}\label{sec:sequences}

In this section we define some relevant terms from Analysis that will be useful in later sections.

\begin{Definition}
A \emph{distance metric} $d$ on a space $W$ is a real-valued function on $W\times W$ with the following properties.

For any $a,b,c\in W$
\begin{enumerate}
	\item $d(a,a) = 0$ 
	\item If $a\neq b$ then $d(a,b)>0$
	\item $d(a,b) = d(b,a)$
	\item $d(a,b) \leq d(a,c) + d(c,b)$.
\end{enumerate}

The pair $(W,d)$ is a \emph{metric space}.
\end{Definition}

% define Cauchy sequences
\begin{Definition}
A sequence $(a_n)$ in a space $W$ is \emph{Cauchy} under a distance metric $d$ if, for any $\epsilon > 0$, there is some $N\in \mathbb{Z}$ such that $\forall n,m>N$, $d(a_n,a_m)<\epsilon$.
\end{Definition}

Next, in order to define the closure of a space, we must first define limit points.

\begin{Definition} 
A \emph{limit point} (or \emph{accumulation point}) $x$ of a metric space $(W,d)$ is an object such that for any $\epsilon>0$, there is at least one element $w\in W\setminus\{x\}$ with $d(w,x)<\epsilon$. 
\end{Definition}

If $d$ is a metric that extends to all limit points of $W$ then we can define the \emph{closure} of $(W,d)$ to be $W$ along with all of its limit points, denoted $\overline{W}$. For example, the closure of $\mathbb{Q}$ under the euclidean metric is $\mathbb{R}$, and since the euclidean metric is defined on the irrationals we can use the same metric on $\mathbb{R}$. 

If $d$ is not defined on the limit points of $W$ then we must define the \emph{completion} of $(W,d)$. Two distinct Cauchy sequences $(a_n), (b_n)$ can get arbitrarily close to each other. Sequences like this, where $d(a_n,b_n) \rightarrow 0$, are called \emph{equivalent sequences}. Equivalence of Cauchy sequences defines an equivalence relation on the space of all Cauchy sequences in $(W,d)$. Note that any element $w\in W$ elicits a trivial equivalence class $[(w)]$ by means of the constant sequence $(w,w,w,\ldots)$.

% real numbers and p-adics as completion of Q under euclidean/p-adic metrics
Consider the set of equivalence classes $\widetilde{W}$ under this equivalence relation. The \emph{completion} of $W$ under $d$ is defined to be $\widetilde{W}$, with distance measured by the asymptotic distance between representative elements of Cauchy sequences. As an example, the rationals $\mathbb{Q}$ can be completed to $\mathbb{Q}_p$ under the $p$-adic metric for some prime $p$.

% why bother with a real-valued distance metric?
We wish to define a distance metric on \setC\ that reflects in some way what real play reveals about similarity among games. The standard game value difference on \setC, $G-H$, is defined to be the game $X$ where $H+X=G$. Instead, we wish to define a real-valued distance.

%We also wish to define a real-valued distance, rather than the standard game value difference on \setC\ wherein $G-H$ is defined to be the game $X$ where $H+X=G$, for future use in computational optimization applications. 

%\section{Distance metrics with high computation requirements}\label{sec:NPmetrics}

% game resolution (should we even mention this? maybe as another possible direction), but it's a natural way to think about distance between games
%	personal communication with aaron siegel
%	resolution similarity = smallest n s.t. G_n != H_n, where G_k = {g_1,g_2,...|g'_1,g'_2,...}, g_i,g'_j are born by day n and g_i <| G <| g'_j
%	examples (e.g. 1/4 and 0 differ on day 0)
%	intractable given how many games born by day n is unknown beyond 4(?) and the numbers grow incredibly fast
%	distance = (1/2)^similarity
%	prove it's a metric
%	examples of Cauchy sequences? Limits are transfinite games at best, as birthdays are actually defined

\section{Diminishing weight distance}\label{sec:dwd}

% diminishing weight distance on $C$, the set of all short games in canonical form
%	given G,H let dwd(G,H) = min_{all edit sequences} sum_{all deletions and additions e} (1/2)^r(e) such that G and H are edited to the same intermediary game digraph, where r(e) is the distance from the root to the source node of e
We are now ready to define our distance metric. 
\begin{Definition}
Let $G,H \in \setC$. An \emph{edit sequence} from $G$ to $H$ is a sequence of colored edge removals and/or additions that turns $D(G)$ into a colored DAG isomorphic to $D(H)$, without changing which vertex is the source. If at any point a vertex becomes isolated then it is removed from the digraph. 

The \emph{cost} of an edit is $\left(\frac{1}{2}\right)^d$, where $d$ is the distance from the source node of the digraph to the initial node of the edge added or removed. The \emph{diminishing weight distance} from $G$ to $H$, denoted $wd(G,H)$, is the minimum total cost over all edit sequences.
\end{Definition}

%	prove it's a metric: definition trivially yields triangle inequality 
\begin{Theorem}
Diminished weight distance is a distance metric on \setC.
\end{Theorem}
\begin{proof}
We must show that $(\setC,wd)$ has the following properties.
\begin{enumerate}
	\item $wd(G,G) = 0$ for all $G\in \setC$
	\item If $G\neq H$ then $wd(G,H)>0$
	\item For all $G,H\in \setC$, $wd(G,H) = wd(H,G)$
	\item For all $G,H,J\in \setC$, $wd(G,H) \leq wd(G,J) + wd(J,H)$.
\end{enumerate}

We prove these in turn.
\begin{enumerate}
	\item If $G\in \setC$ then the lowest total cost edit sequence from $D(G)$ to $D(G)$ is the identity sequence with no edits. Thus $wd(G,G)=0$.
	\item If $G\neq H$ then every edit sequence from $G$ to $H$ requires at least one edit, and edit cost is positive. Hence $wd(G,H)>0$.
	\item Let $S$ be a minimum cost sequence of edits to turn $D(G)$ into $D(H)$. Reversing $S$ and swapping additions and deletions is a minimum cost edit sequence to turn $D(H)$ into $D(G)$. If there was a sequence $S'$ to get from $D(H)$ to $D(G)$ with lower cost then $S$, then its reversal would also turn $D(G)$ to $D(H)$ and have lower cost than $S$.
	\item We finish with the triangle inequality. Consider $G,H,J\in \setC$ and let $S_1,S_2$ be edit sequences that take $D(G)$ to $D(J)$ and $D(J)$ to $D(H)$, respectively. Since each sequence fixes the source vertex of each digraph, the sequence $S_1+S_2$ (the first sequence followed by the second) has total cost equal to the sum of the costs of the two sequences. Thus $S_1+S_2$ is an edit sequence from $D(G)$ to $D(H)$ which cannot have lower cost than $wd(G,H)$. And so $wd(G,H) \leq wd(G,J) + wd(J,H)$.	
\end{enumerate}

\end{proof}

%%%%%%%%%%%%%%%%%%%%%%%

Let's extend the definition of $wd$ to loopy games. Note that the function $D$ is only defined on short games. Let $G$ be a loopy game and $H$ an infinite DAG. The ruleset \textsc{tree} is played on a bicolored digraph with a token on a vertex of the graph (usually beginning at the source node if only one exists). A player's turn consists of moving the token along a directed edge of their color. We say that $G$ is \emph{equivalent to} $H$ ($G\lequiv \textsc{tree}(H)$, abbreviated to $G\lequiv H$) if, for every game $J$, $J+G = J+\textsc{tree}(H)$, using the conventional definition of game addition. If two infinite DAGs, $H_1,H_2$, are both equivalent to the same game $G$ then we say $H_1\lequiv H_2$.  

We first extend $wd$ to include \textsc{tree} played on infinite DAGs. Say $H$ is an infinite DAG and $J\in \setC$. Let $wd(\textsc{tree}(H),J)$ be the minimum cost of an edit sequence taking $H$ to $D(J)$. If instead $J$ is another infinite DAG then let $wd(\textsc{tree}(H),\textsc{tree}(J))$ (which we can abbreviate to $wd(H,J)$) be the minimum cost of an edit sequence taking $H$ to $J$.

Now we can extend the metric $wd$ to loopy games. If $G$ is a loopy game and $J\in \setC$, then $wd(G,J) = \inf_{H_i\lequiv G}\{wd(\textsc{tree}(H_i),J)\}$. That is, $wd(G,J)$ is the infimum over the costs of all edit sequences from $H_i$ to $D(J)$, where $\{H_i\}$ ranges over all infinite DAGs equivalent to $G$. If $J$ is also loopy then $wd(G,J)$ is the infimum ranging over all DAGs equivalent to $G$ and over all those equivalent to $J$. By nature of the geometric series that comprise the resulting infinite edit sequence costs, we can have finite distance between games realized by infinite DAGs.

\iffalse %%%%%%%%%%%%%%%%%%%%%%%
We extend the metric $wd$ to loopy games in the following way. If $G$ is a loopy game and $J\in \setC$, then \craig{$wd(G,J) = \inf_{H_i\lequiv G}\{wd(\textsc{tree}(H_i),J)\}$. That is,} $wd(G,J)$ is the smallest possible cost of an edit sequence from $H_i$ to $D(J)$, where $\{H_i\}$ ranges over all infinite DAGs equivalent to $G$. If $J$ is also loopy then $wd(G,J)$ is the minimum ranging over all DAGs equivalent to $G$ and over all those equivalent to $J$. By nature of the geometric series that comprise the resulting infinite edit sequence costs, we can have finite distance between games realized by infinite DAGs.
\fi %%%%%%%%%%%%%%%%%%%%%%%

%		two sidling sequences to over, mentioned above; show they are cauchy under dwd
For example, consider the two sidling sequences for the game $\textsc{over}=\{0|\textsc{over}\}$, $(a_n) = \left(\frac{1}{2^n}\right)$ and $(b_n) = (*+(\uparrow\!*)\times n)$. First we demonstrate that these sequences are both Cauchy. 

\begin{Theorem}
The sequences $(a_n) = \left(\frac{1}{2^n}\right)$ and $(b_n) = (*+(\uparrow\!*)\times n)$ are Cauchy under $wd$. 
\end{Theorem}
\begin{proof}
For illustration see Figs. \ref{fig:fractions} and \ref{fig:ups}. Fix $N\in \mathbb{N}$ and let $n>m>N$. We can embed $D(\frac{1}{2^m})$ in $D(\frac{1}{2^n})$ by identifying the roots and by matching up the node labeled $1$ in $D(\frac{1}{2^m})$ with the node labeled $\frac{1}{2^m}$ in $D(\frac{1}{2^n})$. Then we can see that 
$$wd\left(\frac{1}{2^n},\frac{1}{2^m}\right) \leq \sum_{i=m}^{n-1} \left(\frac{1}{2^i} + \frac{1}{2^{i+1}}\right) = 3\cdot 2^{-m} - 3\cdot 2^{-n}$$ 
which goes to $0$ as $N\rightarrow \infty$.

An almost identical argument works for $(b_n)$. In this case 
$$wd(b_n,b_m) \leq \underbrace{\frac{1}{2^m}}_\text{rem red to $0$} + \sum_{i=m}^{n-1} \left(\frac{1}{2^i} + \frac{1}{2^{i+1}}\right) + \underbrace{\frac{1}{2^{n}}}_\text{add red to $0$} = 2^{2-m} - 2^{1-n},$$
which again diminishes to $0$ as $N$ increases to infinity. 
\end{proof}

It is obvious that $d(a_n,b_n) = (\frac{1}{2})^{n+1} \rightarrow 0$, hence the sidling sequences are equivalent. We can realize \textsc{over} as an infinite DAG consisting of an infinite red directed path and a blue edge from each node to $0$. 

Thus 
$$wd(a_n,\textsc{over}) \leq \sum_{i=n}^\infty \left(\frac{1}{2^i} + \frac{1}{2^{i+1}}\right) = 3\cdot 2^{-n}$$
which approaches $0$ as $n$ increases. Similarly $wd(b_n,\textsc{over})$ approaches $0$. Hence $\textsc{over}$ is a limit point of \setC.

We continue this section with some more loopy games in the closure of $(\setC,wd)$ as limits of sidling sequences. Since $wd$ is defined on loopy games as well as \setC, we need only find sequences that approach games in order to show they are in $\overline{\setC}$. Any sequence that gets arbitrarily close to a limit point is Cauchy, and hence we need not bother explicitly demonstrating that any of our sequences are Cauchy.

%		on={on|} is limit of the seq (n), and others 
The game $\textsc{on}$, as we discussed above, can be realized as an infinite directed blue path. The digraph $D(n)$, a directed blue path of length $n$, approximates $\textsc{on}$, since $wd(n,\textsc{on}) = \sum_{i=n}^\infty (\frac{1}{2})^i = 2^{1-n} \rightarrow 0$. Of general mathematical interest is the fact that the sequence $(n)$ is Cauchy under $wd$, which differs significantly from the sequence $(n)$ of positive integers under the Euclidean metric.

%              upon$:= \{$upon$|*\}$
Next we consider the game $\textsc{upon}= \{\textsc{upon}|*\}$, in which the Left player can pass and the Right player can move to $*$. For an approximation sequence for this game we first need to introduce \emph{uptimals}.

\begin{Definition}
The game denoted $(x:y)$, called the \emph{ordinal sum} of $x$ and $y$, where $x,y\in \setC$, allows Left (alternatively Right) to play to $(x:y^L)$ or to $x^L$ (alternatively $(x:y^R)$ or $x^R$), where $g^P$ is an option from $g$ for player $P$. The game $\uparrow ^n = (*:n) - (*:(n-1))$. Additionally, the game $\uparrow^{[n]} = (*:n)-* = \sum_{i=1}^n \uparrow^i$. This latter game is also sometimes denoted $\uparrow^{\rightarrow n}$.
\end{Definition}

The obvious negatives of these hold: $-\uparrow^n = \downarrow_n$ and $-\uparrow^{[n]} = -\downarrow_{[n]}$. For further discussion of uptimals and their motivation from the ruleset $\textsc{Hackenbush}$ see page 93 of \cite{SiegelCGT:2013}. For our purposes, we only need that the canonical form of $\uparrow^n = \{0|\downarrow_{[n-1]}\!*\}$, and that $\uparrow^{[n]} = \{\uparrow^{[n-1]}|*\}$. Sidling $\textsc{upon}$ from below yields the following sequence, as demonstrated in \cite{SiegelCGT:2013}.

$$0,\{0|*\}=\uparrow, \{\uparrow|*\}=\uparrow + \uparrow^2, \{\uparrow + \uparrow^2|*\} = \uparrow + \uparrow^2 + \uparrow^3, \ldots = (\uparrow^{[n]})$$                

See Fig. \ref{fig:uptimals} for the first few non-trivial terms of this sequence.

\begin{figure}[ht]
\centering
\begin{minipage}{0.24\textwidth}
\centering
\resizebox{!}{5cm}{% Original Game: {{|}|{{|}|{|}}}
\def\DownBendAngle{40}
\def\UpBendAngle{45}
\def\SourceBendAngle{70}

\begin{tikzpicture}[>=Stealth, auto, node distance=2cm, line width=1.2pt]
\tikzset{
  L_edge/.style={->, very thick, blue},
  R_edge/.style={->, very thick, red},
  loop_large/.style={loop, in=135, out=45, looseness=20, distance=50pt}
}
\node[circle, inner sep=1pt, minimum size=4pt, draw=black, line width=0.5pt, fill=white] (N1) at (0.0, 6.0) {};
\node[circle, inner sep=1pt, minimum size=4pt, draw=black, line width=0.5pt, fill=white] (N2) at (0.0, 0.0) {};
\node[circle, inner sep=1pt, minimum size=4pt, draw=black, line width=0.5pt, fill=white] (N3) at (0.0, 3.0) {};
\draw[L_edge] (N1) to[bend right=\DownBendAngle] (N2);
\draw[R_edge] (N1) to[bend left=\DownBendAngle] (N3);
\draw[L_edge] (N3) to[bend right=\DownBendAngle] (N2);
\draw[R_edge] (N3) to[bend left=\DownBendAngle] (N2);
\end{tikzpicture}}
\end{minipage}%
\begin{minipage}{0.24\textwidth}
  \centering
\resizebox{!}{5cm}{% Original Game: {{{|}|{{|}|{|}}}|{{|}|{|}}}
\def\DownBendAngle{40}
\def\UpBendAngle{45}
\def\SourceBendAngle{70}

\begin{tikzpicture}[>=Stealth, auto, node distance=2cm, line width=1.2pt]
\tikzset{
  L_edge/.style={->, very thick, blue},
  R_edge/.style={->, very thick, red},
  loop_large/.style={loop, in=135, out=45, looseness=20, distance=50pt}
}
\node[circle, inner sep=1pt, minimum size=4pt, draw=black, line width=0.5pt, fill=white] (N1) at (0.0, 9.0) {};
\node[circle, inner sep=1pt, minimum size=4pt, draw=black, line width=0.5pt, fill=white] (N2) at (0.0, 6.0) {};
\node[circle, inner sep=1pt, minimum size=4pt, draw=black, line width=0.5pt, fill=white] (N3) at (0.0, 0.0) {};
\node[circle, inner sep=1pt, minimum size=4pt, draw=black, line width=0.5pt, fill=white] (N4) at (0.0, 3.0) {};
\draw[L_edge] (N1) to[bend right=\DownBendAngle] (N2);
\draw[R_edge] (N1) to[bend left=\DownBendAngle] (N4);
\draw[L_edge] (N2) to[bend right=\DownBendAngle] (N3);
\draw[R_edge] (N2) to[bend left=\DownBendAngle] (N4);
\draw[L_edge] (N4) to[bend right=\DownBendAngle] (N3);
\draw[R_edge] (N4) to[bend left=\DownBendAngle] (N3);
\end{tikzpicture}}
\end{minipage}%
\begin{minipage}{0.24\textwidth}
  \centering
\resizebox{!}{5cm}{% Original Game: {{{{|}|{{|}|{|}}}|{{|}|{|}}}|{{|}|{|}}}
\def\DownBendAngle{40}
\def\UpBendAngle{45}
\def\SourceBendAngle{70}

\begin{tikzpicture}[>=Stealth, auto, node distance=2cm, line width=1.2pt]
\tikzset{
  L_edge/.style={->, very thick, blue},
  R_edge/.style={->, very thick, red},
  loop_large/.style={loop, in=135, out=45, looseness=20, distance=50pt}
}
\node[circle, inner sep=1pt, minimum size=4pt, draw=black, line width=0.5pt, fill=white] (N1) at (0.0, 12.0) {};
\node[circle, inner sep=1pt, minimum size=4pt, draw=black, line width=0.5pt, fill=white] (N2) at (0.0, 9.0) {};
\node[circle, inner sep=1pt, minimum size=4pt, draw=black, line width=0.5pt, fill=white] (N3) at (0.0, 6.0) {};
\node[circle, inner sep=1pt, minimum size=4pt, draw=black, line width=0.5pt, fill=white] (N4) at (0.0, 0.0) {};
\node[circle, inner sep=1pt, minimum size=4pt, draw=black, line width=0.5pt, fill=white] (N5) at (0.0, 3.0) {};
\draw[L_edge] (N1) to[bend right=\DownBendAngle] (N2);
\draw[R_edge] (N1) to[bend left=\DownBendAngle] (N5);
\draw[L_edge] (N2) to[bend right=\DownBendAngle] (N3);
\draw[R_edge] (N2) to[bend left=\DownBendAngle] (N5);
\draw[L_edge] (N3) to[bend right=\DownBendAngle] (N4);
\draw[R_edge] (N3) to[bend left=\DownBendAngle] (N5);
\draw[L_edge] (N5) to[bend right=\DownBendAngle] (N4);
\draw[R_edge] (N5) to[bend left=\DownBendAngle] (N4);
\end{tikzpicture}}
\end{minipage}%
\begin{minipage}{0.24\textwidth}
  \centering
\resizebox{!}{5cm}{% Original Game: {{{{{|}|{{|}|{|}}}|{{|}|{|}}}|{{|}|{|}}}|{{|}|{|}}}
\def\DownBendAngle{40}
\def\UpBendAngle{45}
\def\SourceBendAngle{70}

\begin{tikzpicture}[>=Stealth, auto, node distance=2cm, line width=1.2pt]
\tikzset{
  L_edge/.style={->, very thick, blue},
  R_edge/.style={->, very thick, red},
  loop_large/.style={loop, in=135, out=45, looseness=20, distance=50pt}
}
\node[circle, inner sep=1pt, minimum size=4pt, draw=black, line width=0.5pt, fill=white] (N1) at (0.0, 15.0) {};
\node[circle, inner sep=1pt, minimum size=4pt, draw=black, line width=0.5pt, fill=white] (N2) at (0.0, 12.0) {};
\node[circle, inner sep=1pt, minimum size=4pt, draw=black, line width=0.5pt, fill=white] (N3) at (0.0, 9.0) {};
\node[circle, inner sep=1pt, minimum size=4pt, draw=black, line width=0.5pt, fill=white] (N4) at (0.0, 6.0) {};
\node[circle, inner sep=1pt, minimum size=4pt, draw=black, line width=0.5pt, fill=white] (N5) at (0.0, 0.0) {};
\node[circle, inner sep=1pt, minimum size=4pt, draw=black, line width=0.5pt, fill=white] (N6) at (0.0, 3.0) {};
\draw[L_edge] (N1) to[bend right=\DownBendAngle] (N2);
\draw[R_edge] (N1) to[bend left=\DownBendAngle] (N6);
\draw[L_edge] (N2) to[bend right=\DownBendAngle] (N3);
\draw[R_edge] (N2) to[bend left=\DownBendAngle] (N6);
\draw[L_edge] (N3) to[bend right=\DownBendAngle] (N4);
\draw[R_edge] (N3) to[bend left=\DownBendAngle] (N6);
\draw[L_edge] (N4) to[bend right=\DownBendAngle] (N5);
\draw[R_edge] (N4) to[bend left=\DownBendAngle] (N6);
\draw[L_edge] (N6) to[bend right=\DownBendAngle] (N5);
\draw[R_edge] (N6) to[bend left=\DownBendAngle] (N5);
\end{tikzpicture}}
\end{minipage}%
\caption{The games $\uparrow, \uparrow^{[2]}, \uparrow^{[3]}, \uparrow^{[4]}$}
\label{fig:uptimals}
\end{figure}
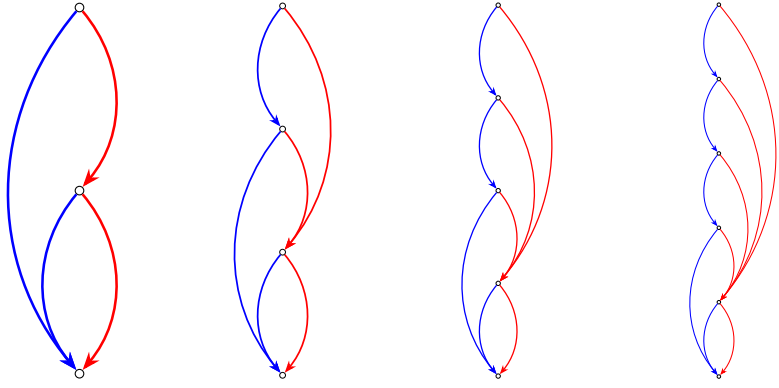

%		Prop 2.11, Siegel shows that upon$=up^{[on]}$ is the sup of $up^{[n]}$ and that $up^{on}$ is the inf of $up^n$
Just as we did with terms in the approximation sequences for $\textsc{over}$, we can embed each term in this sequence as a subgraph of any later term in the sequence. Consider \textsc{upon} to be realized by the infinite DAG that matches every graph in the sequence $(\uparrow^{[n]})$ up to but discluding the edges directed to $0$. In other words, an infinite red path with blue edges from all nodes to $*$.

Given $n\in \mathbb{N}$ we can see that 

$$wd(\uparrow^{[n]},\textsc{upon}) \leq \underbrace{\frac{1}{2^{n-1}}}_\text{rem blue edge to $0$} + \underbrace{\sum_{i=n}^\infty \left(\frac{1}{2^{i-1}} + \frac{1}{2^{i}} \right)}_\text{add blue edge to new node, red edge to $*$} = 2^{3-n}$$ 

which goes to $0$ as $n\rightarrow \infty$. Hence $\textsc{upon}$ is the limit of the sequence $(\uparrow^{[n]})$ and is in the closure of $(\setC,wd)$. We can similarly show that $\uparrow^{\textsc{on}}=\{0|\downarrow_{[on]}\!*\}$ is the limit of the sequence $(\uparrow^{n})$.

%	NP hard since we have to consider all edit sequences (?)

\subsection{Breaking away from sidling sequences}
% Sidling works for finding sides of some stopper-sided games
As noted above the sidling process often yields an infinite sequence of games, which we have seen can converge under the metric $wd$ to certain loopy games. Let's examine some non-stopper loopy games as limit points of Cauchy sequences in $(\setC,wd)$.
%We now turn our attention to Cauchy sequences in $(\setC,wd)$ that are not sidling sequences. 
%In particular, sidling sequences can only approximate stoppers, whether a given game from both the $\textsc{on}$-side and the $\textsc{off}$-side, or the two distinct stoppers that are the sides of a non-stopper.  [[citation?]]

%	non-equivalent cauchy sequences that converge to the same loopy game
%		{n+1|n} and (n)
%		This gives a coarser equivalence class partition of Cauchy sequences in C
Consider the sequence $(\sigma_n) = (\{n+1|n\})$, where each term is a switch. It's easy to see that this sequence is Cauchy, as the difference between terms is a blue path.

$$wd(\sigma_n,\sigma_m) \leq \sum_{i=m}^n \frac{1}{2^i}= 2^{1-m}-2^{-n} \leq 2^{1-m} \rightarrow 0$$

The mean of the terms increases without bound, and in fact the terms approach a loopy game that plays exactly like $\textsc{on}$, i.e. the infinite DAG $\sigma$ that $(\sigma_n)$ approaches is equivalent to \textsc{on}. If we extend the blue path indefinitely, we see the limit point is the game $\{\textsc{on}|\textsc{on}\}$ which is equivalent to $\textsc{on}$. However, the sequence $(\sigma_n)$ is clearly not equivalent to the sequence $(n)$ as Cauchy sequences which also approximates $\textsc{on}$. Therefore, the choice between studying the closure of $(\setC,wd)$ by adding the limit points, or the completion of $(\setC,wd)$ by considering equivalence classes of Cauchy sequences, is an important distinction. This is uncommon in metric spaces and is the result of the process of simplification of combinatorial games.

%	{over|under} from inside (ups and down) and outside (fractions), isn't sidling (is {ups|downs} cauchy?
The switch $\{\textsc{over}|\textsc{under}\}$, composed of the infinitesimal stoppers \textsc{over} and $\textsc{under} = -\textsc{over}$, is also a stopper. While we could consider sidling sequences for $\{\textsc{over}|\textsc{under}\}$, it is more interesting to consider a non-sidling sequence that converges to $\{\textsc{over}|\textsc{under}\}$.

We have already seen that the sequence $(*+(\uparrow\!*)\times n)$ converges to $\textsc{over}$. Now consider the sequence of switches $$(\{*+(\uparrow\!*)\times n|*+(\downarrow\!*)\times n\}) = (\zeta_n).$$

Since the DAG $D(\zeta_n)$ does not collapse under consideration of the canonical form, we need only examine whether or not the sequence converges to our target loopy game. We have already seen how the sequence $(*+(\uparrow\!*)\times n)$ approaches $\textsc{over}$ (and thus similarly how $(*+(\downarrow\!*)\times n)$ approaches $\textsc{under}$) and so we know that

$$wd(\zeta_n,\{\textsc{over}|\textsc{under}\})\leq 2\cdot wd(*+(\uparrow\!*)\times n,\textsc{over}) \rightarrow 0.$$

Hence $(\zeta_n)$ is a sequence that converges to a loopy game neither from above nor below. Every term of $(\zeta_n)$ is incomparable to $\{\textsc{over}|\textsc{under}\}$.

%     $\chi = \{on |\{\chi |off \} \}$, with sides on & +-on
Next we consider an interesting non-stopper. The game $\textsc{off}=-\textsc{on}$, so let $\chi = \{\textsc{on} |\{\chi |\textsc{off} \} \}$. In Fig. \ref{fig:chis}a $\chi$ is the upper left vertex. Although Left can win moving first or whenever they have a turn from $\chi$, if Right moves first then the game alternates forever. This loopy game can be realized as an infinite DAG with alternating moves wherein any two successive turns by either player leads to an infinite path of their color.

Recall that $\pm1 = \{1|-1\}$ and let $(\chi_n)$ be the sequence 
%$$\pm 1, \{2||\pm 1|-2\}, \{3||||2||\pm1|-2|||-3\}, \{4||{3||||2||\pm1|-2|||-3}|-4\},\ldots$$
$$\pm 1, \{3||\pm 1|-2\}, \{5||||3||\pm 1|-2|||-4\}, \ldots$$
We can write this as a recurrence relation.
$$\chi_1 = \pm 1, \chi_{n+1} = \{2n+1|\{\chi_n|-2n\}\}$$

As $n$ increases gameplay in $\chi_n$ approximates play in the loopy game $\chi$. Fig. \ref{fig:chis} shows a few terms of this sequence. 

Now let's show that $(\chi_n)\rightarrow \chi$.

\begin{align*}
wd(\chi_n,\chi)  \leq   &
\underbrace{\frac{2}{2^{2n-1}}}_{\text{rem edges to $0$}}
+ \underbrace{\frac{1}{2^{2n-2}}}_{\text{rem edge to $-1$}} + \\
 &\sum_{i=n}^\infty
\Bigl(
  \underbrace{\frac{1}{2^{2i-2}}+\frac{1}{2^{2i-1}}}_{\text{new alternating edges}}
  + \underbrace{\frac{1}{2^{2i-1}}+\frac{1}{2^{2i}}}_{\text{new edges to integers}}
  + \underbrace{\frac{2}{2^{2i-1}}+\frac{2}{2^{2i}}}_{\text{new integer edges}}
\Bigr)
\end{align*}

which simplifies to $2^{3-2n}+5\cdot 4^{1-n}$. This goes to $0$ as $n\rightarrow \infty$, and hence $\chi \in \overline{\setC}$.

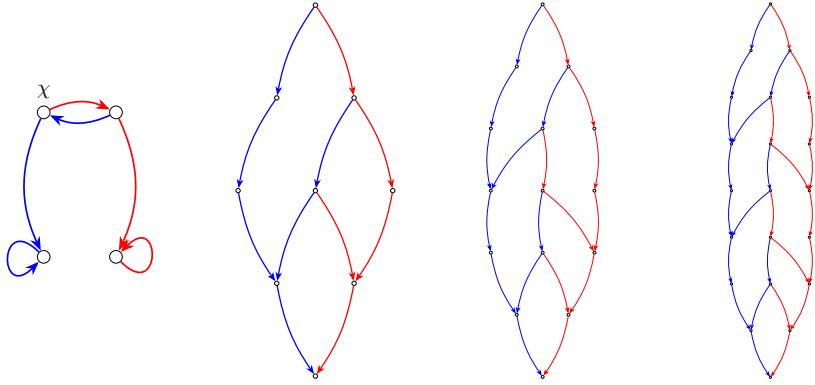
\begin{figure}[ht]
\centering
%\begin{minipage}{0.24\textwidth}
%  \centering
%  \resizebox{!}{5cm}{\input{chi.tikz}}
%\end{minipage}%
\begin{minipage}{0.24\textwidth}
  \centering
  \resizebox{!}{3cm}{%
\begin{tikzpicture}[>=Stealth,scale=1.2]
  % nodes
  \node[circle,draw,fill=white,minimum size=6pt,inner sep=0pt] (a) [label=above:{$\chi$}] at (0,0) {};
  \node[circle,draw,fill=white,minimum size=6pt,inner sep=0pt] (b) at (1,0) {};
  \node[circle,draw,fill=white,minimum size=6pt,inner sep=0pt] (c) at (0,-2) {};
  \node[circle,draw,fill=white,minimum size=6pt,inner sep=0pt] (d) at (1,-2) {};
  % horizontal edges (bent in opposite directions)
  \draw[->,red,thick,bend left=25] (a) to (b);
  \draw[->,blue,thick,bend left=25] (b) to (a);
  % diagonal edges
  \draw[->,blue,thick,bend right=25] (a) to (c);
  \draw[->,red,thick,bend left=25] (b) to (d);
  % self-loops
  \draw[->,blue,thick,looseness=8,min distance=8mm,out=135,in=225] (c) to (c);
  \draw[->,red,thick,looseness=8,min distance=8mm,out=-45,in=45] (d) to (d);
\end{tikzpicture}
  }
\end{minipage}%
\begin{minipage}{0.24\textwidth}
  \centering
  \resizebox{!}{5cm}{% Original Game: {{{{{|}|}|}|}|{{{{|}|}|{|{|}}}|{|{|{|}}}}}
\def\DownBendAngle{12}
\def\UpBendAngle{45}
\def\SourceBendAngle{70}

\begin{tikzpicture}[>=Stealth, auto, node distance=2cm, line width=1.2pt]
\tikzset{
  L_edge/.style={->, very thick, blue},
  R_edge/.style={->, very thick, red},
  loop_large/.style={loop, in=135, out=45, looseness=20, distance=50pt}
}
\node[circle, inner sep=1pt, minimum size=4pt, draw=black, line width=0.5pt, fill=white] (N1) at (0.0, 12.0) {};
\node[circle, inner sep=1pt, minimum size=4pt, draw=black, line width=0.5pt, fill=white] (N2) at (-1.25, 9.0) {};
\node[circle, inner sep=1pt, minimum size=4pt, draw=black, line width=0.5pt, fill=white] (N3) at (-2.5, 6.0) {};
\node[circle, inner sep=1pt, minimum size=4pt, draw=black, line width=0.5pt, fill=white] (N4) at (-1.25, 3.0) {};
\node[circle, inner sep=1pt, minimum size=4pt, draw=black, line width=0.5pt, fill=white] (N5) at (0.0, 0.0) {};
\node[circle, inner sep=1pt, minimum size=4pt, draw=black, line width=0.5pt, fill=white] (N6) at (1.25, 9.0) {};
\node[circle, inner sep=1pt, minimum size=4pt, draw=black, line width=0.5pt, fill=white] (N7) at (0.0, 6.0) {};
\node[circle, inner sep=1pt, minimum size=4pt, draw=black, line width=0.5pt, fill=white] (N8) at (1.25, 3.0) {};
\node[circle, inner sep=1pt, minimum size=4pt, draw=black, line width=0.5pt, fill=white] (N9) at (2.5, 6.0) {};
\draw[L_edge] (N1) to[bend right=\DownBendAngle] (N2);
\draw[R_edge] (N1) to[bend left=\DownBendAngle] (N6);
\draw[L_edge] (N2) to[bend right=\DownBendAngle] (N3);
\draw[L_edge] (N3) to[bend right=\DownBendAngle] (N4);
\draw[L_edge] (N4) to[bend right=\DownBendAngle] (N5);
\draw[L_edge] (N6) to[bend right=\DownBendAngle] (N7);
\draw[R_edge] (N6) to[bend left=\DownBendAngle] (N9);
\draw[L_edge] (N7) to[bend right=\DownBendAngle] (N4);
\draw[R_edge] (N7) to[bend left=\DownBendAngle] (N8);
\draw[R_edge] (N8) to[bend left=\DownBendAngle] (N5);
\draw[R_edge] (N9) to[bend left=\DownBendAngle] (N8);
\end{tikzpicture}}
\end{minipage}%
\begin{minipage}{0.24\textwidth}
  \centering
  \resizebox{!}{5cm}{% Original Game: {{{{{{{|}|}|}|}|}|}|{{{{{{|}|}|}|}|{{{{|}|}|{|{|}}}|{|{|{|}}}}}|{|{|{|{|{|}}}}}}}
\def\DownBendAngle{12}
\def\UpBendAngle{45}
\def\SourceBendAngle{70}

\begin{tikzpicture}[>=Stealth, auto, node distance=2cm, line width=1.2pt]
\tikzset{
  L_edge/.style={->, very thick, blue},
  R_edge/.style={->, very thick, red},
  loop_large/.style={loop, in=135, out=45, looseness=20, distance=50pt}
}
\node[circle, inner sep=1pt, minimum size=4pt, draw=black, line width=0.5pt, fill=white] (N1) at (0.0, 18.0) {};
\node[circle, inner sep=1pt, minimum size=4pt, draw=black, line width=0.5pt, fill=white] (N2) at (-1.25, 15.0) {};
\node[circle, inner sep=1pt, minimum size=4pt, draw=black, line width=0.5pt, fill=white] (N3) at (-2.5, 12.0) {};
\node[circle, inner sep=1pt, minimum size=4pt, draw=black, line width=0.5pt, fill=white] (N4) at (-2.5, 9.0) {};
\node[circle, inner sep=1pt, minimum size=4pt, draw=black, line width=0.5pt, fill=white] (N5) at (-2.5, 6.0) {};
\node[circle, inner sep=1pt, minimum size=4pt, draw=black, line width=0.5pt, fill=white] (N6) at (-1.25, 3.0) {};
\node[circle, inner sep=1pt, minimum size=4pt, draw=black, line width=0.5pt, fill=white] (N7) at (0.0, 0.0) {};
\node[circle, inner sep=1pt, minimum size=4pt, draw=black, line width=0.5pt, fill=white] (N8) at (1.25, 15.0) {};
\node[circle, inner sep=1pt, minimum size=4pt, draw=black, line width=0.5pt, fill=white] (N9) at (0.0, 12.0) {};
\node[circle, inner sep=1pt, minimum size=4pt, draw=black, line width=0.5pt, fill=white] (N10) at (0.0, 9.0) {};
\node[circle, inner sep=1pt, minimum size=4pt, draw=black, line width=0.5pt, fill=white] (N11) at (0.0, 6.0) {};
\node[circle, inner sep=1pt, minimum size=4pt, draw=black, line width=0.5pt, fill=white] (N12) at (1.25, 3.0) {};
\node[circle, inner sep=1pt, minimum size=4pt, draw=black, line width=0.5pt, fill=white] (N13) at (2.5, 6.0) {};
\node[circle, inner sep=1pt, minimum size=4pt, draw=black, line width=0.5pt, fill=white] (N14) at (2.5, 12.0) {};
\node[circle, inner sep=1pt, minimum size=4pt, draw=black, line width=0.5pt, fill=white] (N15) at (2.5, 9.0) {};
\draw[L_edge] (N1) to[bend right=\DownBendAngle] (N2);
\draw[R_edge] (N1) to[bend left=\DownBendAngle] (N8);
\draw[L_edge] (N2) to[bend right=\DownBendAngle] (N3);
\draw[L_edge] (N3) to[bend right=\DownBendAngle] (N4);
\draw[L_edge] (N4) to[bend right=\DownBendAngle] (N5);
\draw[L_edge] (N5) to[bend right=\DownBendAngle] (N6);
\draw[L_edge] (N6) to[bend right=\DownBendAngle] (N7);
\draw[L_edge] (N8) to[bend right=\DownBendAngle] (N9);
\draw[R_edge] (N8) to[bend left=\DownBendAngle] (N14);
\draw[L_edge] (N9) to[bend right=\DownBendAngle] (N4);
\draw[R_edge] (N9) to[bend left=\DownBendAngle] (N10);
\draw[L_edge] (N10) to[bend right=\DownBendAngle] (N11);
\draw[R_edge] (N10) to[bend left=\DownBendAngle] (N13);
\draw[L_edge] (N11) to[bend right=\DownBendAngle] (N6);
\draw[R_edge] (N11) to[bend left=\DownBendAngle] (N12);
\draw[R_edge] (N12) to[bend left=\DownBendAngle] (N7);
\draw[R_edge] (N13) to[bend left=\DownBendAngle] (N12);
\draw[R_edge] (N14) to[bend left=\DownBendAngle] (N15);
\draw[R_edge] (N15) to[bend left=\DownBendAngle] (N13);
\end{tikzpicture}}
\end{minipage}%
\begin{minipage}{0.24\textwidth}
  \centering
  \resizebox{!}{5cm}{% Original Game: {{{{{{{{{|}|}|}|}|}|}|}|}|{{{{{{{{|}|}|}|}|}|}|{{{{{{|}|}|}|}|{{{{|}|}|{|{|}}}|{|{|{|}}}}}|{|{|{|{|{|}}}}}}}|{|{|{|{|{|{|{|}}}}}}}}}
\def\DownBendAngle{12}
\def\UpBendAngle{45}
\def\SourceBendAngle{70}

\begin{tikzpicture}[>=Stealth, auto, node distance=2cm, line width=1.2pt]
\tikzset{
  L_edge/.style={->, very thick, blue},
  R_edge/.style={->, very thick, red},
  loop_large/.style={loop, in=135, out=45, looseness=20, distance=50pt}
}
\node[circle, inner sep=1pt, minimum size=4pt, draw=black, line width=0.5pt, fill=white] (N1) at (0.0, 24.0) {};
\node[circle, inner sep=1pt, minimum size=4pt, draw=black, line width=0.5pt, fill=white] (N2) at (-1.25, 21.0) {};
\node[circle, inner sep=1pt, minimum size=4pt, draw=black, line width=0.5pt, fill=white] (N3) at (-2.5, 18.0) {};
\node[circle, inner sep=1pt, minimum size=4pt, draw=black, line width=0.5pt, fill=white] (N4) at (-2.5, 15.0) {};
\node[circle, inner sep=1pt, minimum size=4pt, draw=black, line width=0.5pt, fill=white] (N5) at (-2.5, 12.0) {};
\node[circle, inner sep=1pt, minimum size=4pt, draw=black, line width=0.5pt, fill=white] (N6) at (-2.5, 9.0) {};
\node[circle, inner sep=1pt, minimum size=4pt, draw=black, line width=0.5pt, fill=white] (N7) at (-2.5, 6.0) {};
\node[circle, inner sep=1pt, minimum size=4pt, draw=black, line width=0.5pt, fill=white] (N8) at (-1.25, 3.0) {};
\node[circle, inner sep=1pt, minimum size=4pt, draw=black, line width=0.5pt, fill=white] (N9) at (0.0, 0.0) {};
\node[circle, inner sep=1pt, minimum size=4pt, draw=black, line width=0.5pt, fill=white] (N10) at (1.25, 21.0) {};
\node[circle, inner sep=1pt, minimum size=4pt, draw=black, line width=0.5pt, fill=white] (N11) at (0.0, 18.0) {};
\node[circle, inner sep=1pt, minimum size=4pt, draw=black, line width=0.5pt, fill=white] (N12) at (0.0, 15.0) {};
\node[circle, inner sep=1pt, minimum size=4pt, draw=black, line width=0.5pt, fill=white] (N13) at (0.0, 12.0) {};
\node[circle, inner sep=1pt, minimum size=4pt, draw=black, line width=0.5pt, fill=white] (N14) at (0.0, 9.0) {};
\node[circle, inner sep=1pt, minimum size=4pt, draw=black, line width=0.5pt, fill=white] (N15) at (0.0, 6.0) {};
\node[circle, inner sep=1pt, minimum size=4pt, draw=black, line width=0.5pt, fill=white] (N16) at (1.25, 3.0) {};
\node[circle, inner sep=1pt, minimum size=4pt, draw=black, line width=0.5pt, fill=white] (N17) at (2.5, 6.0) {};
\node[circle, inner sep=1pt, minimum size=4pt, draw=black, line width=0.5pt, fill=white] (N18) at (2.5, 12.0) {};
\node[circle, inner sep=1pt, minimum size=4pt, draw=black, line width=0.5pt, fill=white] (N19) at (2.5, 9.0) {};
\node[circle, inner sep=1pt, minimum size=4pt, draw=black, line width=0.5pt, fill=white] (N20) at (2.5, 18.0) {};
\node[circle, inner sep=1pt, minimum size=4pt, draw=black, line width=0.5pt, fill=white] (N21) at (2.5, 15.0) {};
\draw[L_edge] (N1) to[bend right=\DownBendAngle] (N2);
\draw[R_edge] (N1) to[bend left=\DownBendAngle] (N10);
\draw[L_edge] (N2) to[bend right=\DownBendAngle] (N3);
\draw[L_edge] (N3) to[bend right=\DownBendAngle] (N4);
\draw[L_edge] (N4) to[bend right=\DownBendAngle] (N5);
\draw[L_edge] (N5) to[bend right=\DownBendAngle] (N6);
\draw[L_edge] (N6) to[bend right=\DownBendAngle] (N7);
\draw[L_edge] (N7) to[bend right=\DownBendAngle] (N8);
\draw[L_edge] (N8) to[bend right=\DownBendAngle] (N9);
\draw[L_edge] (N10) to[bend right=\DownBendAngle] (N11);
\draw[R_edge] (N10) to[bend left=\DownBendAngle] (N20);
\draw[L_edge] (N11) to[bend right=\DownBendAngle] (N4);
\draw[R_edge] (N11) to[bend left=\DownBendAngle] (N12);
\draw[L_edge] (N12) to[bend right=\DownBendAngle] (N13);
\draw[R_edge] (N12) to[bend left=\DownBendAngle] (N18);
\draw[L_edge] (N13) to[bend right=\DownBendAngle] (N6);
\draw[R_edge] (N13) to[bend left=\DownBendAngle] (N14);
\draw[L_edge] (N14) to[bend right=\DownBendAngle] (N15);
\draw[R_edge] (N14) to[bend left=\DownBendAngle] (N17);
\draw[L_edge] (N15) to[bend right=\DownBendAngle] (N8);
\draw[R_edge] (N15) to[bend left=\DownBendAngle] (N16);
\draw[R_edge] (N16) to[bend left=\DownBendAngle] (N9);
\draw[R_edge] (N17) to[bend left=\DownBendAngle] (N16);
\draw[R_edge] (N18) to[bend left=\DownBendAngle] (N19);
\draw[R_edge] (N19) to[bend left=\DownBendAngle] (N17);
\draw[R_edge] (N20) to[bend left=\DownBendAngle] (N21);
\draw[R_edge] (N21) to[bend left=\DownBendAngle] (N18);
\end{tikzpicture}}
\end{minipage}%
\caption{The loopy game $\chi$ and the $2^{nd}, 3^{rd}$, and $4^{th}$ terms of an approximation sequence}
\label{fig:chis}
\end{figure}

Let's add another non-stopper to the closure of $(\setC,wd)$. Consider the game $\psi = \{0,\psi|0,\psi\}$ as seen in Fig. \ref{fig:psis}a.  

%        	$\psi = \{0,pass |0,pass \} (bowtie) with sides ^[on]* & v[on]*
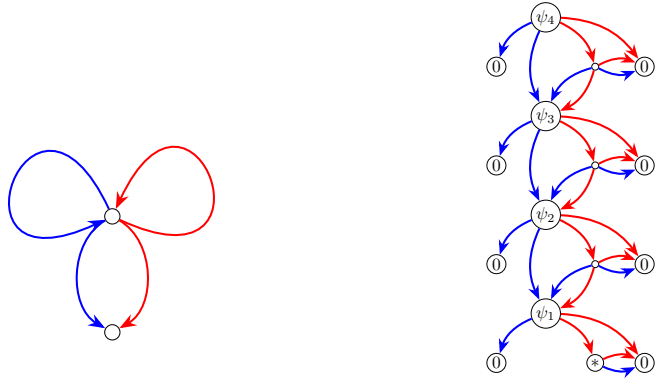
\begin{figure}[ht]
\centering
\begin{minipage}{0.48\textwidth}
  \centering
  \resizebox{!}{4cm}{%
  \begin{tikzpicture}[scale=3.2,>=Stealth]
  \node[circle,draw,fill=white,minimum size=6pt,inner sep=0pt] (A) at (0,0) {};
  \node[circle,draw,fill=white,minimum size=6pt,inner sep=0pt] (B) at (0,-.5) {};
  \draw[->,blue,thick]
    (A) edge[out=115,in=205,loop,looseness=8,min distance=8mm] (A);
  \draw[->,red,thick]
    (A) edge[out=-25,in=65,loop,looseness=8,min distance=8mm] (A);
  \draw[->,red,thick,bend left=60] (A) to (B);
  \draw[->,blue,thick,bend right=60] (A) to (B);
\end{tikzpicture}
  } 
\end{minipage}%
\begin{minipage}{0.48\textwidth}
  \centering
  \resizebox{!}{5cm}{% Original Game: {0,psi3|{0,psi3|0,psi3},0}

\begin{tikzpicture}[>=Stealth, auto, node distance=2cm, line width=1.2pt]
\tikzset{
  L_edge/.style={->, very thick, blue},
  R_edge/.style={->, very thick, red},
  loop_large/.style={loop, in=135, out=45, looseness=20, distance=50pt}
}
%\node[circle, inner sep=1pt, minimum size=4pt, draw=black, line width=0.5pt, fill=white] (0) at (0.0, 0,0) {$0$};
\node[circle, inner sep=1pt, minimum size=4pt, draw=black, line width=0.5pt, fill=white] (0) at (-1.0, 1,0) {$0$};
\node[circle, inner sep=1pt, minimum size=4pt, draw=black, line width=0.5pt, fill=white] (star) at (1.0, 1) {$*$};
\node[circle, inner sep=1pt, minimum size=4pt, draw=black, line width=0.5pt, fill=white] (psi1) at (0.0, 2.0) {$\psi_1$};
\node[circle, inner sep=1pt, minimum size=4pt, draw=black, line width=0.5pt, fill=white]
(01R) at (2.0, 1) {$0$};

\node[circle, inner sep=1pt, minimum size=4pt, draw=black, line width=0.5pt, fill=white]
(psi2) at (0.0, 4.0) {$\psi_2$};
\node[circle, inner sep=1pt, minimum size=4pt, draw=black, line width=0.5pt, fill=white]
(psi22) at (1.0, 3) {};
\node[circle, inner sep=1pt, minimum size=4pt, draw=black, line width=0.5pt, fill=white]
(02R) at (2.0, 3) {$0$};
\node[circle, inner sep=1pt, minimum size=4pt, draw=black, line width=0.5pt, fill=white]
(02L) at (-1.0, 3) {$0$};

\node[circle, inner sep=1pt, minimum size=4pt, draw=black, line width=0.5pt, fill=white]
(psi3) at (0.0, 6.0) {$\psi_3$};
\node[circle, inner sep=1pt, minimum size=4pt, draw=black, line width=0.5pt, fill=white]
(psi32) at (1.0, 5) {};
\node[circle, inner sep=1pt, minimum size=4pt, draw=black, line width=0.5pt, fill=white]
(03R) at (2.0, 5) {$0$};
\node[circle, inner sep=1pt, minimum size=4pt, draw=black, line width=0.5pt, fill=white]
(03L) at (-1.0, 5) {$0$};

\node[circle, inner sep=1pt, minimum size=4pt, draw=black, line width=0.5pt, fill=white]
(psi4) at (0.0,8.0) {$\psi_4$};
\node[circle, inner sep=1pt, minimum size=4pt, draw=black, line width=0.5pt, fill=white]
(psi42) at (1.0, 7) {};
\node[circle, inner sep=1pt, minimum size=4pt, draw=black, line width=0.5pt, fill=white]
(04R) at (2.0, 7) {$0$};
\node[circle, inner sep=1pt, minimum size=4pt, draw=black, line width=0.5pt, fill=white]
(04L) at (-1.0, 7) {$0$};

\draw[L_edge] (psi4) to[bend right=25] (psi3);
\draw[L_edge] (psi4) to[bend right=25] (04L);
\draw[R_edge] (psi4) to[bend left=25] (psi42);
\draw[R_edge] (psi4) to[bend left=25] (04R);
\draw[L_edge] (psi42) to[bend right=25] (psi3);
\draw[R_edge] (psi42) to[bend left=25] (psi3);
\draw[L_edge] (psi42) to[bend right=25] (04R);
\draw[R_edge] (psi42) to[bend left=25] (04R);

\draw[L_edge] (psi3) to[bend right=25] (psi2);
\draw[L_edge] (psi3) to[bend right=25] (03L);
\draw[R_edge] (psi3) to[bend left=25] (psi32);
\draw[R_edge] (psi3) to[bend left=25] (03R);
\draw[L_edge] (psi32) to[bend right=25] (psi2);
\draw[R_edge] (psi32) to[bend left=25] (psi2);
\draw[L_edge] (psi32) to[bend right=25] (03R);
\draw[R_edge] (psi32) to[bend left=25] (03R);
 
\draw[L_edge] (psi2) to[bend right=25] (psi1);
\draw[L_edge] (psi2) to[bend right=25] (02L);
\draw[R_edge] (psi2) to[bend left=25] (psi22);
\draw[R_edge] (psi2) to[bend left=25] (02R);
\draw[L_edge] (psi22) to[bend right=25] (psi1);
\draw[R_edge] (psi22) to[bend left=25] (psi1);
\draw[L_edge] (psi22) to[bend right=25] (02R);
\draw[R_edge] (psi22) to[bend left=25] (02R);

\draw[L_edge] (psi1) to[bend right=25] (0);
\draw[R_edge] (psi1) to[bend left=25] (star);
\draw[R_edge] (psi1) to[bend left=25] (01R);
%\draw[R_edge] (star) to[bend left=25] (0);
%\draw[L_edge] (star) to[bend right=25] (0);
\draw[R_edge] (star) to[bend left=25] (01R);
\draw[L_edge] (star) to[bend right=25] (01R);
\end{tikzpicture}}
\end{minipage}%
\caption{The game $\psi$ and a term of an approximation sequence. Some nodes are deidentified for clarity.}
\label{fig:psis}
\end{figure}

The game $\psi$ can be approximated by the recurrence relation:
$$\psi_1 = \downarrow\!* = \{0|*,0\}, \psi_{n+1} = \{\psi_n,0|\{\psi_n,0|\psi_n,0\},0\}$$

For any $n$ all non-trivial followers of $\psi_n$ have both Left and Right options to $0$, as well as Left and Right options to positions that themselves have moves to $0$ and to previous terms in the sequence. If we extend the approximation indefinitely by adding blue and red edges from $*$ and a blue edge from $\psi_1$ in Fig. \ref{fig:psis}b to a new $\psi_i$ and continuing the chain of $\psi_i$'s, then we get an infinite DAG equivalent to $\psi$. We then see that for any $n$ the distance between $\psi_n$ and $\psi$ is bounded above by:

$$wd(\psi_n,\psi) \leq \underbrace{\frac{1}{2^{n-1}}}_\text{new blue from $\psi_1$} + \underbrace{\frac{2}{2^n}}_\text{new edges from $*$} + \sum_{i=n}^\infty \underbrace{\left(\frac{4}{2^i} + \frac{4}{2^{i+1}} \right)}_\text{each new block} = 2^{4-n}$$

which approaches $0$ as $n\rightarrow 0$.

%        	$dud = {pass|pass}
We now proceed to the loopy game with the most extreme sides possible, the game $\textsc{dud}=\{\textsc{dud}|\textsc{dud}\}$. In \textsc{dud} both players only have the option to pass, and its sides are \textsc{on} and \textsc{off}. The game is clearly not a stopper, nor is it even an ender as it is impossible for the game to end. Interestingly any game, short or long, added to \textsc{dud} results in \textsc{dud}.

Consider the sequence $(\delta_n)$ defined by the following recurrence relation. 
$$ \delta_1 = \{3|2||1|0\}, \delta_{n+1} = \{\delta_n + F_{n+4}|\delta_n\}$$

where $F_k$ is the $k^{th}$ term in the Fibonacci sequence, beginning with $F_1=F_2=1$. This yields the following first few terms. Repeated values are highlighted in color to show identified nodes in the resulting DAGs.

\[
\hspace*{-1.5in}
\begin{aligned}
\delta_1 &=& \{3|2||1|0\}\\ 
\delta_2 &=& \{6|5||4|\ccyan{3}|||\ccyan{3}|2||1|0\}\\
\delta_3 &=& \{11|10||9|\ccyan{8}|||\ccyan{8}|7||\corange{6|5}||||\corange{6|5}||4|\ccyan{3}|||\ccyan{3}|2||1|0\}\\
%\delta_4 &=& \{\{19|18||17|16|||16|15||14|13||||14|13||12|11|||11|10||9|8\}|\{11|10||9|8|||8|7||6|5||||6|5||4|3|||3|2||1|0\}\}
\end{aligned}
\]

\iffalse %%%%%%%%%%%%%%%%%%%%
\craig{I'll copy the same list but without highlighting repeated parts in case we don't want to include that}

\[
\hspace*{-1.5in}
\begin{aligned}
\delta_1 &=& \{3|2||1|0\}\\ 
\delta_2 &=& \{6|5||4|3|||3|2||1|0\}\\
\delta_3 &=& \{11|10||9|8|||8|7||6|5||||6|5||4|3|||3|2||1|0\}\\
%\delta_4 &=& \{\{19|18||17|16|||16|15||14|13||||14|13||12|11|||11|10||9|8\}|\{11|10||9|8|||8|7||6|5||||6|5||4|3|||3|2||1|0\}\}
\end{aligned}
\]
\fi %%%%%%%%%%%%%%%%%%%%

See Fig. \ref{fig:duds} for an illustration of how the DAGs for this sequence progress.

% preamble: \usepackage{tikz,graphicx}
\begin{figure}[ht]
\centering
\setlength{\unitlength}{1pt} % optional
\newlength{\tikzcommonheight}
\setlength{\tikzcommonheight}{3.5cm} % change this to your desired height

% row 1
\begin{minipage}[b]{0.48\textwidth}
  \centering
  \resizebox{!}{\tikzcommonheight}{%
\begin{tikzpicture}[scale=4.2,>=Stealth]
  % main node
  \node[circle,draw,fill=white,minimum size=6pt,inner sep=0pt] (a) at (0,0) {};
  \draw[->,blue,thick]
    (a) edge[out=135,in=225,loop,looseness=8,min distance=8mm]
    node[above left=1pt] {} (a);
  \draw[->,red,thick]
    (a) edge[out=-45,in=45,loop,looseness=8,min distance=8mm]
    node[below right=1pt] {} (a);
\end{tikzpicture}
 } 
\end{minipage}\hfill
\begin{minipage}[b]{0.48\textwidth}
  \centering
  \resizebox{!}{\tikzcommonheight}{% Original Game: {{3|2}|{1|0}}
\begin{tikzpicture}[>=Stealth, auto, node distance=2cm, line width=1.2pt]
\tikzset{
  L_edge/.style={->, very thick, blue},
  R_edge/.style={->, very thick, red},
  loop_large/.style={loop, in=135, out=45, looseness=20, distance=50pt}
}
\node[circle, inner sep=1pt, minimum size=4pt, draw=black, line width=0.5pt, fill=white] (N1) at (0.0, 6.0) {};
\node[circle, inner sep=1pt, minimum size=4pt, draw=black, line width=0.5pt, fill=white] (N2) at (-1.25, 3.0) {};
\node[circle, inner sep=1pt, minimum size=4pt, draw=black, line width=0.5pt, fill=white] (N3) at (-3.75, 0.0) {};
\node[circle, inner sep=1pt, minimum size=4pt, draw=black, line width=0.5pt, fill=white] (N4) at (-1.25, 0.0) {};
\node[circle, inner sep=1pt, minimum size=4pt, draw=black, line width=0.5pt, fill=white] (N5) at (1.25, 0.0) {};
\node[circle, inner sep=1pt, minimum size=4pt, draw=black, line width=0.5pt, fill=white] (N6) at (3.75, 0.0) {};
\node[circle, inner sep=1pt, minimum size=4pt, draw=black, line width=0.5pt, fill=white] (N7) at (1.25, 3.0) {};
\draw[L_edge] (N1) to[bend right=12] (N2);
\draw[R_edge] (N1) to[bend left=12] (N7);
\draw[L_edge] (N2) to[bend right=12] (N3);
\draw[R_edge] (N2) to[bend left=12] (N4);
\draw[L_edge] (N3) to[bend right=12] (N4);
\draw[L_edge] (N4) to[bend right=12] (N5);
\draw[L_edge] (N5) to[bend right=12] (N6);
\draw[L_edge] (N7) to[bend right=12] (N5);
\draw[R_edge] (N7) to[bend left=12] (N6);
\end{tikzpicture}}
\end{minipage}

\vspace{1em}

% row 2
  \centering
  \resizebox{!}{\tikzcommonheight}{% Original Game: {{{6|5}|{4|3}}|{{3|2}|{1|0}}}
\begin{tikzpicture}[>=Stealth, auto, node distance=2cm, line width=1.2pt]
\tikzset{
  L_edge/.style={->, very thick, blue},
  R_edge/.style={->, very thick, red},
  loop_large/.style={loop, in=135, out=45, looseness=20, distance=50pt}
}
\node[circle, inner sep=1pt, minimum size=4pt, draw=black, line width=0.5pt, fill=white] (N1) at (0.0, 9.0) {};
\node[circle, inner sep=1pt, minimum size=4pt, draw=black, line width=0.5pt, fill=white] (N2) at (-1.25, 6.0) {};
\node[circle, inner sep=1pt, minimum size=4pt, draw=black, line width=0.5pt, fill=white] (N3) at (-3.75, 3.0) {};
\node[circle, inner sep=1pt, minimum size=4pt, draw=black, line width=0.5pt, fill=white] (N4) at (-7.5, 0.0) {};
\node[circle, inner sep=1pt, minimum size=4pt, draw=black, line width=0.5pt, fill=white] (N5) at (-5.0, 0.0) {};
\node[circle, inner sep=1pt, minimum size=4pt, draw=black, line width=0.5pt, fill=white] (N6) at (-2.5, 0.0) {};
\node[circle, inner sep=1pt, minimum size=4pt, draw=black, line width=0.5pt, fill=white] (N7) at (0.0, 0.0) {};
\node[circle, inner sep=1pt, minimum size=4pt, draw=black, line width=0.5pt, fill=white] (N8) at (2.5, 0.0) {};
\node[circle, inner sep=1pt, minimum size=4pt, draw=black, line width=0.5pt, fill=white] (N9) at (5.0, 0.0) {};
\node[circle, inner sep=1pt, minimum size=4pt, draw=black, line width=0.5pt, fill=white] (N10) at (7.5, 0.0) {};
\node[circle, inner sep=1pt, minimum size=4pt, draw=black, line width=0.5pt, fill=white] (N11) at (-1.25, 3.0) {};
\node[circle, inner sep=1pt, minimum size=4pt, draw=black, line width=0.5pt, fill=white] (N12) at (1.25, 6.0) {};
\node[circle, inner sep=1pt, minimum size=4pt, draw=black, line width=0.5pt, fill=white] (N13) at (1.25, 3.0) {};
\node[circle, inner sep=1pt, minimum size=4pt, draw=black, line width=0.5pt, fill=white] (N14) at (3.75, 3.0) {};
\draw[L_edge] (N1) to[bend right=12] (N2);
\draw[R_edge] (N1) to[bend left=12] (N12);
\draw[L_edge] (N2) to[bend right=12] (N3);
\draw[R_edge] (N2) to[bend left=12] (N11);
\draw[L_edge] (N3) to[bend right=12] (N4);
\draw[R_edge] (N3) to[bend left=12] (N5);
\draw[L_edge] (N4) to[bend right=12] (N5);
\draw[L_edge] (N5) to[bend right=12] (N6);
\draw[L_edge] (N6) to[bend right=12] (N7);
\draw[L_edge] (N7) to[bend right=12] (N8);
\draw[L_edge] (N8) to[bend right=12] (N9);
\draw[L_edge] (N9) to[bend right=12] (N10);
\draw[L_edge] (N11) to[bend right=12] (N6);
\draw[R_edge] (N11) to[bend left=12] (N7);
\draw[L_edge] (N12) to[bend right=12] (N13);
\draw[R_edge] (N12) to[bend left=12] (N14);
\draw[L_edge] (N13) to[bend right=12] (N7);
\draw[R_edge] (N13) to[bend left=12] (N8);
\draw[L_edge] (N14) to[bend right=12] (N9);
\draw[R_edge] (N14) to[bend left=12] (N10);
\end{tikzpicture}}

\vspace{1em}

% row 3
  \centering
  \resizebox{\textwidth}{!}{\input{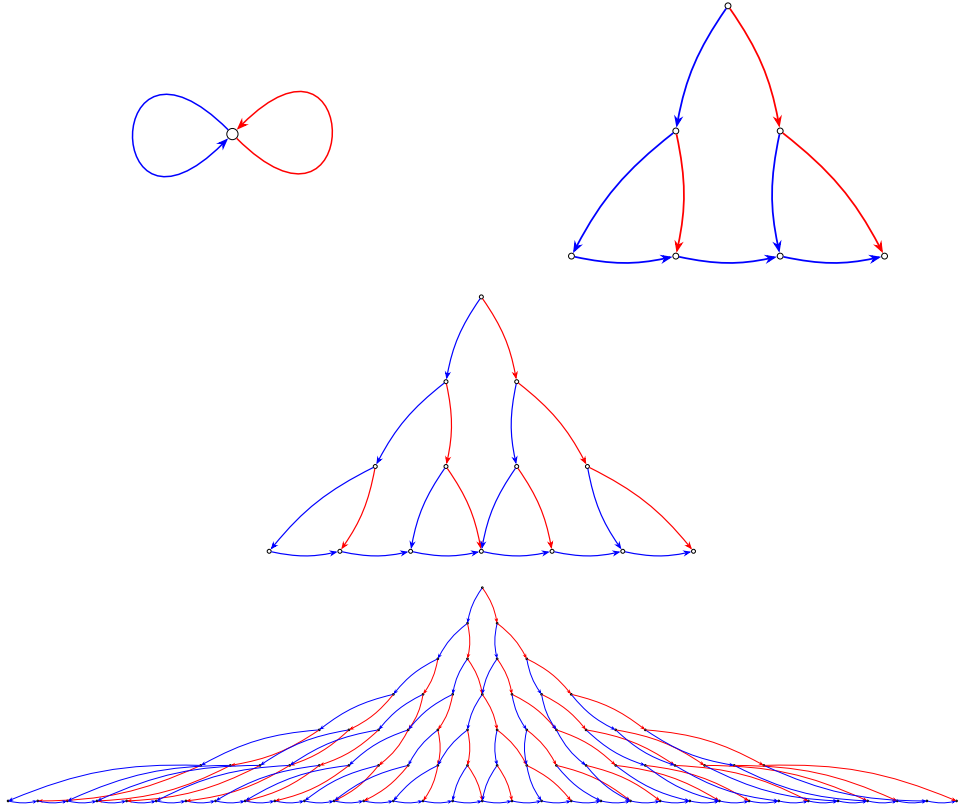}}

\caption{The plumtree \textsc{dud} and terms $1,2$, and $5$ of its approximation sequence}
\label{fig:duds}
\end{figure}

Each term contains all previous terms of the sequence after the removal of the edges representing the final integer values. When $n>m>N\in \mathbb{N}$ we see that

$$wd(\delta_n,\delta_m) \leq \underbrace{\left(\frac{1}{2}\right)^{m+1}\left(F_{m+4}-2\right)}_\text{remove integer path} + \underbrace{\sum_{i=m+1}^n 2\left(F_{i+4}-2\right)\left(\frac{1}{2}\right)^i}_\text{add edges to options} + \underbrace{\left(\frac{1}{2}\right)^{n+1}\left(F_{n+4}-2\right)}_\text{add integer path}.$$

Binet's formula gives the closed formula $F_n = \lfloor \frac{\phi ^n}{\sqrt{5}} \rceil$ where $\phi = \frac{1}{2}(1+\sqrt{5})$ and $\lfloor \cdot \rceil$ indicates nearest integer rounding. This allows us to simplify the distance by bounding $F_k$ above by $\frac{\phi ^{k+1}}{\sqrt{5}}$.

{\small
\[
wd(\delta_n,\delta_m) \leq \left(\frac{1}{2}\right)^{m+1}\left(\frac{\phi ^{m+5}}{\sqrt{5}}-2\right) + \sum_{i=m+1}^n 2\left(\frac{\phi ^{i+5}}{\sqrt{5}}-2\right)\left(\frac{1}{2}\right)^i + \left(\frac{1}{2}\right)^{n+1}\left(\frac{\phi ^{n+5}}{\sqrt{5}}-2\right).
\]
}
 
This is further bounded above by

$$\frac{\phi^4}{\sqrt{5}}\left(\frac{\phi}{2}\right)^{m+1}-2^{-m}+\frac{\phi^4}{\sqrt{5}}\left(\frac{\phi}{2}\right)^{n+1}-2^{-n} + \frac{2\phi^5}{\sqrt{5}} \sum_{i=m+1}^{n}\left(\frac{1}{2}\right)^i.$$

Since $\sum_{k=m+1}^n \left(\frac{1}{2}\right)^k = (2^{-m}-2^{-n})$ we get an upper bound of 

$$\frac{\phi^4}{\sqrt{5}}\left(\frac{\phi}{2}\right)^{m+1}-2^{-m}+\frac{\phi^4}{\sqrt{5}}\left(\frac{\phi}{2}\right)^{n+1}-2^{-n} + \frac{2\phi^5}{\sqrt{5}} \left(2^{-m}-2^{-n}\right).$$

As $\frac{\phi}{2} < 1$, this $\rightarrow 0$ as $N\rightarrow \infty$, and hence the sequence $(\delta_n)$ is Cauchy. 

Now let's examine the structure of the terms in the sequence. At every position in the game $\delta_n$ both Left and Right have an option to another position in the game until one player reaches an integer value, at which point the game stops. One DAG equivalent to \textsc{dud} is in fact an infinite binary tree, which our sequence does not approximate, and due to the rapid growth of binary trees of height $n$ as $n$ increases no such sequence of finite games will be Cauchy under $wd$. However, as $n$ increases, the number of turns in $\delta_n$ until either player reaches an integer also increases. As we saw above with two non-equivalent sequences in \setC\ that approach \textsc{on}, we conclude that the sequence $(\delta_n)$ approximates $\textsc{dud}$. 

Other than approximating a non-stopper, this sequence is interesting in the fact that it can be altered to sit inside different outcome classes. Note that $\forall n, \delta_n \in \mathcal{L}$. However, the sequence $(-\delta_n)$ approximates $\textsc{dud}$ just as well and sits entirely in $\mathcal{R}$. In fact, we can shift the terms and, with a small alteration, ensure they are all in $\mathcal{N}$.

\[
%\hspace*{-.5in}
\begin{alignedat}{3}
\delta_1^\mathcal{N} &=& \{2|1|&|-1|-2|\}\\
\delta_2^\mathcal{N} &=& \{3|2||1|0|&||0|-1||-2|-3\}\\
\delta_3^\mathcal{N} &=& \{6|5||4|3|||3|2||1|-1||&||1|-1||-2|-3|||-3|-4||-5|-6\}
\end{alignedat}
\]

The games in $(\delta_n^\mathcal{N})$ grow in a very similar way to those in $(\delta_n)$, and so we omit the arithmetic showing that this sequence is also Cauchy.

%	non-plumtrees
\subsection{Non plum trees}\label{subsec:non-plum}

Let's turn our attention to non-short games that are not plum trees. 

Bach's Carousel (Fig. \ref{fig:carousel}a) is a rather well-known construction, developed by Clive Bach as discussed in \cite{WinningWays:2001} as an example of a set of non-stopper sided loopy games. We next prove that there is no sequence of games in \setC\ that has a loopy game from Bach's Carousel as a limit point.

\begin{Theorem}
No loopy game in Bach's Carousel is the limit of a Cauchy sequence of games in $(\setC,wd)$. 
\end{Theorem}\label{thm:bachs_carousel}
\begin{proof}
Let us assume that $(g_n)$, a sequence of games in $(\setC,wd)$, gets arbitrarily close to a loopy game in Bach's Carousel. Without loss of generality we may choose any of the four loopy games in the system as a root, so let's assume our short games have $\alpha_1$ as a root, and that this sequence of short games approximates $\alpha$. The first few terms of $\alpha_1$ in any $g_n$ for sufficiently large $n$ must have both Left and Right options to $0$, and a Left option to an approximate of $\beta$, say $\beta_1$. See Fig. \ref{fig:carousel} for a generalization of the terms in the sequence $(g_n)$, where $x$ is some short game. We demonstrate that $\beta_1>0$, and hence $\alpha_1$'s Left option to $0$ is dominated by their option to $\beta_1$. 

If Left moves first on $\beta_1$ then she can move to $1$. If Right moves first, then play will alternate until Right ends their turn on $\alpha_2$, after which Left can move to $0$. Therefore $\beta_1>0$ and Left's option from $\alpha_1$ to $0$ is dominated. Hence, the short games in the sequence $(g_n)$ for sufficiently large $n$ cannot approximate the game $\alpha$ from Bach's Carousel, which requires a Left option from $\alpha$ to $0$.
\end{proof}

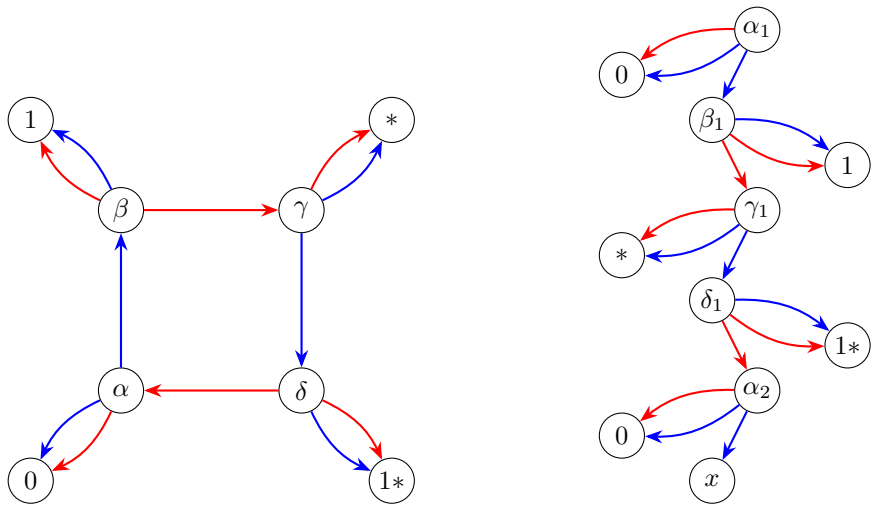
\begin{figure}[ht]
\centering
% --- bach's carousel ---
\begin{minipage}[b]{0.45\textwidth}
\centering
\begin{tikzpicture}[scale=1.2, >=Stealth, every node/.style={circle, draw, minimum size=6mm, inner sep=0pt}]
  \node (1) at (-2,2) {$1$};
  \node (*) at (2,2) {$*$};
  \node (1*) at (2,-2) {$1*$};
  \node (0) at (-2,-2) {$0$};
  \node (alpha) at (-1,-1) {$\alpha$};
  \node (beta) at (-1,1) {$\beta$};
  \node (gamma) at (1,1) {$\gamma$};
  \node (delta) at (1,-1) {$\delta$};

  \draw[->, thick, red] (beta) -- (gamma);
  \draw[->, thick, red] (delta) -- (alpha);

  \draw[->, thick, blue] (alpha) -- (beta);
  \draw[->, thick, blue] (gamma) -- (delta);

  \draw[->, thick, red, bend left=20] (beta) to (1);
  \draw[->, thick, blue, bend right=20] (beta) to (1);

  \draw[->, thick, red, bend left=20] (gamma) to (*);
  \draw[->, thick, blue, bend right=20] (gamma) to (*);

  \draw[->, thick, red, bend left=20] (delta) to (1*);
  \draw[->, thick, blue, bend right=20] (delta) to (1*);

  \draw[->, thick, red, bend left=20] (alpha) to (0);
  \draw[->, thick, blue, bend right=20] (alpha) to (0);
\end{tikzpicture}
\end{minipage}
\hfill
% --- bach's carousel approximation ---
\begin{minipage}[b]{0.45\textwidth}
\centering
\begin{tikzpicture}[scale=1.2, >=Stealth, every node/.style={circle, draw, minimum size=6mm, inner sep=0pt}]
  \node (alpha1) at (0,0) {$\alpha_1$};
  \node (beta1) at (-0.5,-1) {$\beta_1$};
  \node (gamma1) at (0,-2) {$\gamma_1$};
  \node (delta1) at (-0.5,-3) {$\delta_1$};
  \node (alpha2) at (0,-4) {$\alpha_2$};
  \node (x) at (-0.5,-5) {$x$};

  \node (0a) at (-1.5,-0.5) {$0$};
  \node (1a) at (1,-1.5) {$1$};
  \node (*a) at (-1.5,-2.5) {$*$};
  \node (1*a) at (1,-3.5) {$1*$};
  \node (0b) at (-1.5,-4.5) {$0$};

  \draw[->, thick, blue, bend left=20] (alpha1) to (0a);
  \draw[->, thick, red, bend right=20] (alpha1) to (0a);
  \draw[->, thick, blue] (alpha1) -- (beta1);

  \draw[->, thick, blue, bend left=20] (beta1) to (1a);
  \draw[->, thick, red, bend right=20] (beta1) to (1a);
  \draw[->, thick, red] (beta1) -- (gamma1);

  \draw[->, thick, blue, bend left=20] (gamma1) to (*a);
  \draw[->, thick, red, bend right=20] (gamma1) to (*a);
  \draw[->, thick, blue] (gamma1) -- (delta1);

  \draw[->, thick, blue, bend left=20] (delta1) to (1*a);
  \draw[->, thick, red, bend right=20] (delta1) to (1*a);
  \draw[->, thick, red] (delta1) -- (alpha2);

  \draw[->, thick, blue, bend left=20] (alpha2) to (0b);
  \draw[->, thick, red, bend right=20] (alpha2) to (0b);
  \draw[->, thick, blue] (alpha2) -- (x);
\end{tikzpicture}
\end{minipage}
\caption{Bach's Carousel and how a term of its approximation sequence must appear, if it exists. We have labeled a number of non-zero sink vertices in the images to clarify the options from each position.}
\label{fig:carousel}
\end{figure}

%		tis={tisn|} and tisn={|tis} are not limits of any cauchy sequence, but they are stopper-sided
A much simpler pair of non-stopper loopy games that are not limit points of sequences in $(\setC,wd)$ are \textsc{tis} and \textsc{tisn}. As defined in \cite{WinningWays:2001}, $\textsc{tis}=\{\textsc{tisn}|\}$ and $\textsc{tisn}=\{|\textsc{tis}\}$ (Fig. \ref{fig:tis_tisn}). 

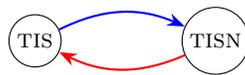
\begin{figure}[ht]
\centering
\begin{tikzpicture}[>=Stealth,scale=1.2]
  \node[circle,draw,fill=white,inner sep=2pt] (tis) at (-1,0) {\textsc{tis}};
  \node[circle,draw,fill=white,inner sep=2pt] (tisn) at (1,0) {\textsc{tisn}};
  \draw[->,thick,blue,bend left=25] (tis) to (tisn);
  \draw[->,thick,red,bend left=25] (tisn) to (tis);
\end{tikzpicture}
\caption{\textsc{tis} and \textsc{tisn}}
\label{fig:tis_tisn}
\end{figure}

\begin{Theorem}
There is no Cauchy sequence in $(\setC,wd)$ that gets arbitrarily close to either \textsc{tis} or \textsc{tisn}.
\end{Theorem}
\begin{proof}
If $G$ is a short game in \setC\ that approximates \textsc{tis} then it has no option for Right. But any short game with only Left options simplifies to an integer. Hence no sequence in $(\setC,wd)$ approaches \textsc{tis}. A similar argument shows the same is true for \textsc{tisn}. However, the sidling process shows that the sides of $\textsc{tis}$ are $1 \& 0$, and for $\textsc{tisn}$ are $0 \& -1$, which are not only stoppers but in \setC.
\end{proof}

%		bach1 and bach 2 sequences
%		The question remains whether or not the sum bach1 + bach2 cauchy sequences converge to the sum bach1 + bach2, which is non stopper-sided

We finish this section with a couple infinite but loop-free games.

On page 370 of \cite{WinningWays:2001}, along with the discussion of Bach's Carousel, two loop-free infinite games also discovered by Clive Bach are introduced. We deem these games $B_0^1$ and $B_0^2$ (Fig. \ref{fig:bachsum}) after their discoverer. Note that $\frac{*}{2}$, called a \emph{semi-star}, is shorthand for the game $\{*,\uparrow|\downarrow *,0\}$, and has the interesting (but not unique) property that $(\frac{*}{2} + \frac{*}{2}) = *$.

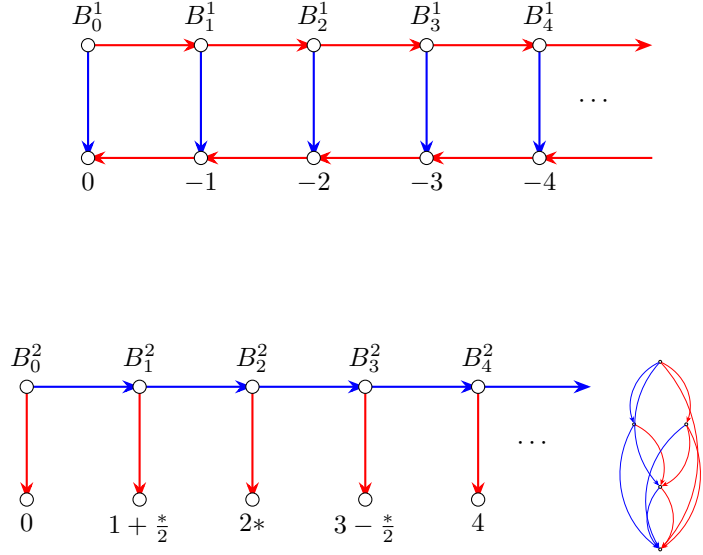
\begin{figure}
\centering

% bach1
\begin{tikzpicture}[scale=1.5, >=Stealth]
  \tikzset{dot/.style={draw, circle, fill=white, minimum size=5pt, inner sep=0pt}}

  \coordinate (b0) at (0,1);
  \coordinate (b1) at (1,1);
  \coordinate (b2) at (2,1);
  \coordinate (b3) at (3,1);
  \coordinate (b4) at (4,1);
  \coordinate (b5) at (5,1);

  \coordinate (0) at (0,0);
  \coordinate (-1) at (1,0);
  \coordinate (-2) at (2,0);
  \coordinate (-3) at (3,0);
  \coordinate (-4) at (4,0);
  \coordinate (-5) at (5,0);

  \node[above=2pt of b0] {$B_0^1$};
  \node[above=2pt of b1] {$B_1^1$};
  \node[above=2pt of b2] {$B_2^1$};
  \node[above=2pt of b3] {$B_3^1$};
  \node[above=2pt of b4] {$B_4^1$};

  \node[below=2pt of 0] {$0$};
  \node[below=2pt of -1] {$-1$};
  \node[below=2pt of -2] {$-2$};
  \node[below=2pt of -3] {$-3$};
  \node[below=2pt of -4] {$-4$};

  \node at (4.5,0.5) {$\cdots$};

  \draw[->, thick, red] (b0) -- (b1);
  \draw[->, thick, red] (b1) -- (b2);
  \draw[->, thick, red] (b2) -- (b3);
  \draw[->, thick, red] (b3) -- (b4);
  \draw[->, thick, red] (b4) -- (b5);

  \draw[->, thick, blue] (b0) -- (0);
  \draw[->, thick, blue] (b1) -- (-1);
  \draw[->, thick, blue] (b2) -- (-2);
  \draw[->, thick, blue] (b3) -- (-3);
  \draw[->, thick, blue] (b4) -- (-4);

  \draw[->, thick, red] (-5) -- (-4);
  \draw[->, thick, red] (-4) -- (-3);
  \draw[->, thick, red] (-3) -- (-2);
  \draw[->, thick, red] (-2) -- (-1);
  \draw[->, thick, red] (-1) -- (0);

  \foreach \p in {b0,b1,b2,b3,b4,0,-1,-2,-3,-4}
    \node[dot] at (\p) {};
\end{tikzpicture}

\vspace{5em}

% bach2
\begin{tikzpicture}[scale=1.5, >=Stealth]
  \tikzset{dot/.style={draw, circle, fill=white, minimum size=5pt, inner sep=0pt}}

  \coordinate (b0) at (0,1);
  \coordinate (b1) at (1,1);
  \coordinate (b2) at (2,1);
  \coordinate (b3) at (3,1);
  \coordinate (b4) at (4,1);
  \coordinate (b5) at (5,1);

  \coordinate (0) at (0,0);
  \coordinate (1) at (1,0);
  \coordinate (2) at (2,0);
  \coordinate (3) at (3,0);
  \coordinate (4) at (4,0);
  \coordinate (5) at (5,0);

  \node[above=2pt of b0] {$B_0^2$};
  \node[above=2pt of b1] {$B_1^2$};
  \node[above=2pt of b2] {$B_2^2$};
  \node[above=2pt of b3] {$B_3^2$};
  \node[above=2pt of b4] {$B_4^2$};

  \node[below=2pt of 0] {$0$};
  \node[below=2pt of 1] {$1+\frac{*}{2}$};
  \node[below=2pt of 2] {$2*$};
  \node[below=2pt of 3] {$3-\frac{*}{2}$};
  \node[below=2pt of 4] {$4$};

  \node at (4.5,0.5) {$\cdots$};

  \draw[->, thick, blue] (b0) -- (b1);
  \draw[->, thick, blue] (b1) -- (b2);
  \draw[->, thick, blue] (b2) -- (b3);
  \draw[->, thick, blue] (b3) -- (b4);
  \draw[->, thick, blue] (b4) -- (b5);

  \draw[->, thick, red] (b0) -- (0);
  \draw[->, thick, red] (b1) -- (1);
  \draw[->, thick, red] (b2) -- (2);
  \draw[->, thick, red] (b3) -- (3);
  \draw[->, thick, red] (b4) -- (4);

  \foreach \p in {b0,b1,b2,b3,b4,0,1,2,3,4}
    \node[dot] at (\p) {};
\end{tikzpicture}
  \resizebox{!}{1in}{% Original Game: {{{|}|{|}},{{|}|{{|}|{|}}}|{|},{{|}|{|},{{|}|{|}}}}
\def\DownBendAngle{42}
\def\UpBendAngle{45}
\def\SourceBendAngle{70}

\begin{tikzpicture}[>=Stealth, auto, node distance=2cm, line width=1.2pt]
\tikzset{
  L_edge/.style={->, very thick, blue},
  R_edge/.style={->, very thick, red},
  loop_large/.style={loop, in=135, out=45, looseness=20, distance=50pt}
}
\node[circle, inner sep=1pt, minimum size=4pt, draw=black, line width=0.5pt, fill=white] (N1) at (0.0, 9.0) {};
\node[circle, inner sep=1pt, minimum size=4pt, draw=black, line width=0.5pt, fill=white] (N2) at (-1.25, 6.0) {};
\node[circle, inner sep=1pt, minimum size=4pt, draw=black, line width=0.5pt, fill=white] (N3) at (0.0, 0.0) {};
\node[circle, inner sep=1pt, minimum size=4pt, draw=black, line width=0.5pt, fill=white] (N4) at (0.0, 3.0) {};
\node[circle, inner sep=1pt, minimum size=4pt, draw=black, line width=0.5pt, fill=white] (N5) at (1.25, 6.0) {};
\draw[L_edge] (N1) to[bend right=\DownBendAngle] (N2);
\draw[L_edge] (N1) to[bend right=\DownBendAngle] (N4);
\draw[R_edge] (N1) to[bend left=\DownBendAngle] (N5);
\draw[R_edge] (N1) to[bend left=\DownBendAngle] (N3);
\draw[L_edge] (N2) to[bend right=\DownBendAngle] (N3);
\draw[R_edge] (N2) to[bend left=\DownBendAngle] (N4);
\draw[L_edge] (N4) to[bend right=\DownBendAngle] (N3);
\draw[R_edge] (N4) to[bend left=\DownBendAngle] (N3);
\draw[L_edge] (N5) to[bend right=\DownBendAngle] (N3);
\draw[R_edge] (N5) to[bend left=\DownBendAngle] (N3);
\draw[R_edge] (N5) to[bend left=\DownBendAngle] (N4);
\end{tikzpicture}}
\caption{Two infinite loop-free games and the DAG $D(\frac{*}{2})$. We omit from the second figure the visual mess of arrows required for $0$ to be the only sink.}
\label{fig:bachsum}
\end{figure}

Bach notes that although both $B_0^1$ and $B_0^2$ are stoppers, their sum $B=(B_0^1 + B_0^2)$ is not only not a stopper but is not even stopper-sided. We can approximate both of these games with the sequences $(H_n^1)$ and $(H_n^2)$ defined in the following way.

\[
\hspace*{-2in}
\begin{aligned}
H_1^1 &= *, H_{n+1}^1 &=& \{0|H_n^1-1\} \\
H_1^2 &= *, H_{n+1}^2 &=& \{H_n^2+1+\frac{*}{2}|0\} \\
\end{aligned}
\]

\iffalse %%%%%%%%%%%%%%%%%%%%%%%%%
Both of these approximation sequences amount to cutting off the respective infinite game and ending with both Left and Right options to a positive or negative value. The distances between the $m^{th}$ and $n^{th}$ terms, $n>m>N$, under the $wd$ metric are easy to calculate.

\hspace*{-1cm}%
\begin{tabular}{r c c c c c c}
$wd(H_n^1,H_m^1)$ & $\leq$ & 
$\underbrace{\frac{1}{2^m}}_{\text{rem edge to $-m$}}$ & $+$ &
$\underbrace{\frac{2}{2^n} + \frac{1}{2^n}}_{\text{add edges near $-n$}}$ & $+$ &
$\underbrace{\sum_{i=m+1}^{n-1}\frac{3}{2^i}}_{\text{extend DAG}}$ \\[1ex]

$wd(H_n^2,H_m^2)$ & $\leq$ &
$\underbrace{\frac{2}{2^m}}_{\text{rem end edges}}$ & $+$ &
$\underbrace{\frac{2}{2^n} + \frac{1}{2^n}}_{\text{add edges at end}}$ & $+$ &
$\underbrace{\sum_{i=m+1}^{n-1} \left(\frac{4}{2^{i+1}} + \frac{5}{2^{i+2}} + \frac{2}{2^{i+3}}\right)}_{\text{over-estimate of edges in $\frac{*}{2}$}}$ \\
\end{tabular}

These expressions simplify to 

\[
\hspace*{-2in}
\begin{aligned}
wd(H_n^1,H_m^1) &\leq 2^{2-m}-2^{2-n}+2^{-n}\\ 
wd(H_n^2,H_m^2) &\leq 7\cdot 2^{-m-1}+2^{1-m}-2^{2-n}
\end{aligned}
\]

both of which approach $0$ as $N\rightarrow \infty$.

\craig{replace the cauchy arithmetic with convergence arithmetic}
\fi %%%%%%%%%%%%%%%%%%%%%%%%

Both of these approximation sequences amount to cutting off the respective infinite game and ending with both Left and Right options to a positive or negative value. The distances between the $n^{th}$ terms and their loopy target under the $wd$ metric are easy to calculate.

%\hspace*{-1cm}%
{\small
\begin{tabular}{r c c c c c l}
$wd(H_n^1,B_0^1)$ & $\leq$ & 
$\underbrace{\frac{1}{2^n}}_{\text{edge to $-n$}}$ & $+$ &
$\underbrace{\sum_{i=n+1}^{\infty}\frac{3}{2^i}}_{\text{extend DAG}}$ & $=$ & $\dfrac{1}{2^{n-2}}$ \\[1ex]

$wd(H_n^2,B_0^2)$ & $\leq$ &
$\underbrace{\frac{2}{2^n}}_{\text{end edges}}$ & $+$ &
$\sum\limits_{i=n+1}^{\infty} \underbrace{\left(\frac{4}{2^{i+1}} + \frac{5}{2^{i+2}} + \frac{2}{2^{i+3}}\right)}_{\text{over-estimate of edges in $\frac{*}{2}$}}$ & $=$ & $\dfrac{7}{2^{n+1}}+\dfrac{1}{2^{n-1}}$\\
\end{tabular}
}

both of which approach $0$ as $n\rightarrow \infty$.

%%%%%%%%%%%%%%%%%%%%%%%%%%%%%

While the sum of Cauchy sequences is Cauchy under the usual addition operation, since we are using game addition this is not necessarily a given. As a counterexample, consider the following sequences in \setC.

$$(a_n) = \left(1,1,\frac{1}{2},\frac{1}{2},\frac{1}{4},\frac{1}{4},\ldots \right), 
(b_n) = \left(-1,-\frac{1}{2},-\frac{1}{2},-\frac{1}{4},-\frac{1}{4},\ldots \right)
$$

Both sequences are Cauchy since the sequence $\left(\frac{1}{2}\right)^n$ is Cauchy. However, their sum $(a_n+b_n) = (0,\frac{1}{2},0,\frac{1}{4},0,\frac{1}{8},\ldots)$ is not, since it alternates between positive values and $0$ which differ under $wd$ by at least $1$. 

Unfortunately the question of whether or not the sum of the terms $(H_n^1+H_n^2)$ is Cauchy itself remains elusive. Fig. \ref{fig:bach_sum} shows the first few terms of $(H_n^1+H_n^2)$. While the sheer density of the digraphs appears intimidating enough to ward off all but the bravest graph theorists, we already see by the third term a third node appear as an out-neighbor of the source vertex. If new nodes continue to appear close to the source vertex, then it is unlikely that the sequence is Cauchy.

\begin{figure}[ht]
\centering
\begin{minipage}{0.32\textwidth}
  \centering
  \resizebox{!}{5cm}{% Original Game: {{{{{|}|},{{{|}|},{{{|}|}|{{|}|}}|{{|}|}}|{{{|}|}|{{|}|}},{{{{|}|}|{{|}|}}|{{|}|}}}|{|}},{{{{|}|},{{{|}|},{{{|}|}|{{|}|}}|{{|}|}}|{{{|}|}|{{|}|}},{{{{|}|}|{{|}|}}|{{|}|}}}|{{{|}|{|}},{{|}|{{|}|{|}}}|{|},{{|}|{|},{{|}|{|}}}}}|{{|}|{{|{|}}|{|{|}}}},{{{{|}|{|}},{{|}|{{|}|{|}}}|{|},{{|}|{|},{{|}|{|}}}}|{{|{|}}|{|{|}}}}}
\def\DownBendAngle{12}
\def\UpBendAngle{45}
\def\SourceBendAngle{70}

\begin{tikzpicture}[>=Stealth, auto, node distance=2cm, line width=1.2pt]
\tikzset{
  L_edge/.style={->, very thick, blue!70!black},
  R_edge/.style={->, very thick, red!70!black},
  loop_large/.style={loop, in=135, out=45, looseness=20, distance=50pt}
}
\node[circle, inner sep=1pt, minimum size=4pt, draw=black, line width=0.5pt, fill=white] (N1) at (0.0, 18.0) {};
\node[circle, inner sep=1pt, minimum size=4pt, draw=black, line width=0.5pt, fill=white] (N2) at (1.25, 15.0) {};
\node[circle, inner sep=1pt, minimum size=4pt, draw=black, line width=0.5pt, fill=white] (N3) at (-1.25, 12.0) {};
\node[circle, inner sep=1pt, minimum size=4pt, draw=black, line width=0.5pt, fill=white] (N4) at (1.25, 9.0) {};
\node[circle, inner sep=1pt, minimum size=4pt, draw=black, line width=0.5pt, fill=white] (N5) at (-3.75, 6.0) {};
\node[circle, inner sep=1pt, minimum size=4pt, draw=black, line width=0.5pt, fill=white] (N6) at (-2.5, 3.0) {};
\node[circle, inner sep=1pt, minimum size=4pt, draw=black, line width=0.5pt, fill=white] (N7) at (0.0, 0.0) {};
\node[circle, inner sep=1pt, minimum size=4pt, draw=black, line width=0.5pt, fill=white] (N8) at (-3.75, 9.0) {};
\node[circle, inner sep=1pt, minimum size=4pt, draw=black, line width=0.5pt, fill=white] (N9) at (3.75, 9.0) {};
\node[circle, inner sep=1pt, minimum size=4pt, draw=black, line width=0.5pt, fill=white] (N10) at (-1.25, 6.0) {};
\node[circle, inner sep=1pt, minimum size=4pt, draw=black, line width=0.5pt, fill=white] (N11) at (0.0, 3.0) {};
\node[circle, inner sep=1pt, minimum size=4pt, draw=black, line width=0.5pt, fill=white] (N12) at (3.75, 6.0) {};
\node[circle, inner sep=1pt, minimum size=4pt, draw=black, line width=0.5pt, fill=white] (N13) at (-1.25, 15.0) {};
\node[circle, inner sep=1pt, minimum size=4pt, draw=black, line width=0.5pt, fill=white] (N14) at (1.25, 12.0) {};
\node[circle, inner sep=1pt, minimum size=4pt, draw=black, line width=0.5pt, fill=white] (N15) at (1.25, 6.0) {};
\node[circle, inner sep=1pt, minimum size=4pt, draw=black, line width=0.5pt, fill=white] (N16) at (2.5, 3.0) {};
\node[circle, inner sep=1pt, minimum size=4pt, draw=black, line width=0.5pt, fill=white] (N17) at (-1.25, 9.0) {};
\draw[L_edge] (N1) to[bend right=\DownBendAngle] (N2);
\draw[L_edge] (N1) to[bend right=\DownBendAngle] (N13);
\draw[R_edge] (N1) to[bend left=\DownBendAngle] (N14);
\draw[R_edge] (N1) to[bend left=\DownBendAngle] (N17);
\draw[L_edge] (N2) to[bend right=\DownBendAngle] (N3);
\draw[R_edge] (N2) to[bend left=\DownBendAngle] (N9);
\draw[L_edge] (N3) to[bend right=\DownBendAngle] (N4);
\draw[L_edge] (N3) to[bend right=\DownBendAngle] (N6);
\draw[R_edge] (N3) to[bend left=\DownBendAngle] (N8);
\draw[R_edge] (N3) to[bend left=\DownBendAngle] (N5);
\draw[L_edge] (N4) to[bend right=\DownBendAngle] (N5);
\draw[L_edge] (N4) to[bend right=\DownBendAngle] (N6);
\draw[R_edge] (N4) to[bend left=\DownBendAngle] (N6);
\draw[L_edge] (N5) to[bend right=\DownBendAngle] (N6);
\draw[R_edge] (N5) to[bend left=\DownBendAngle] (N6);
\draw[L_edge] (N6) to[bend right=\DownBendAngle] (N7);
\draw[L_edge] (N8) to[bend right=\DownBendAngle] (N5);
\draw[R_edge] (N8) to[bend left=\DownBendAngle] (N6);
\draw[L_edge] (N9) to[bend right=\DownBendAngle] (N10);
\draw[L_edge] (N9) to[bend right=\DownBendAngle] (N11);
\draw[R_edge] (N9) to[bend left=\DownBendAngle] (N12);
\draw[R_edge] (N9) to[bend left=\DownBendAngle] (N7);
\draw[L_edge] (N10) to[bend right=\DownBendAngle] (N7);
\draw[R_edge] (N10) to[bend left=\DownBendAngle] (N11);
\draw[L_edge] (N11) to[bend right=\DownBendAngle] (N7);
\draw[R_edge] (N11) to[bend left=\DownBendAngle] (N7);
\draw[L_edge] (N12) to[bend right=\DownBendAngle] (N7);
\draw[R_edge] (N12) to[bend left=\DownBendAngle] (N7);
\draw[R_edge] (N12) to[bend left=\DownBendAngle] (N11);
\draw[L_edge] (N13) to[bend right=\DownBendAngle] (N3);
\draw[R_edge] (N13) to[bend left=\DownBendAngle] (N7);
\draw[L_edge] (N14) to[bend right=\DownBendAngle] (N9);
\draw[R_edge] (N14) to[bend left=\DownBendAngle] (N15);
\draw[L_edge] (N15) to[bend right=\DownBendAngle] (N16);
\draw[R_edge] (N15) to[bend left=\DownBendAngle] (N16);
\draw[R_edge] (N16) to[bend left=\DownBendAngle] (N7);
\draw[L_edge] (N17) to[bend right=\DownBendAngle] (N7);
\draw[R_edge] (N17) to[bend left=\DownBendAngle] (N15);
\end{tikzpicture}}
\end{minipage}%
\begin{minipage}{0.32\textwidth}
  \centering
  \resizebox{!}{5cm}{% Original Game: {{{{{{|}|}|}|{{{{|}|}|{{|}|}},{{{|}|}|{{{|}|}|{{|}|}}}|{{|}|},{{{|}|}|{{|}|},{{{|}|}|{{|}|}}}}}|{|}},{{{{{|}|}|}|{{{{|}|}|{{|}|}},{{{|}|}|{{{|}|}|{{|}|}}}|{{|}|},{{{|}|}|{{|}|},{{{|}|}|{{|}|}}}}}|{{{{|}|}|{{|}|{|}}},{{{|}|}|{{{|}|{|}},{{|}|{{|}|{|}}}|{|},{{|}|{|},{{|}|{|}}}}}|{{{|}|{|}}|{{|{|}},{{|{|}},{{|{|}}|{|{|}}}|{|{|}}}|{{|{|}}|{|{|}}},{{{|{|}}|{|{|}}}|{|{|}}}}},{{{{|}|{|}},{{|}|{{|}|{|}}}|{|},{{|}|{|},{{|}|{|}}}}|{{|{|}},{{|{|}},{{|{|}}|{|{|}}}|{|{|}}}|{{|{|}}|{|{|}}},{{{|{|}}|{|{|}}}|{|{|}}}}}}}|{{|}|{{|{|}}|{{|{|{|}}}|{|{|{|}}}}}},{{{{{|}|}|{{|}|{|}}},{{{|}|}|{{{|}|{|}},{{|}|{{|}|{|}}}|{|},{{|}|{|},{{|}|{|}}}}}|{{{|}|{|}}|{{|{|}},{{|{|}},{{|{|}}|{|{|}}}|{|{|}}}|{{|{|}}|{|{|}}},{{{|{|}}|{|{|}}}|{|{|}}}}},{{{{|}|{|}},{{|}|{{|}|{|}}}|{|},{{|}|{|},{{|}|{|}}}}|{{|{|}},{{|{|}},{{|{|}}|{|{|}}}|{|{|}}}|{{|{|}}|{|{|}}},{{{|{|}}|{|{|}}}|{|{|}}}}}}|{{|{|}}|{{|{|{|}}}|{|{|{|}}}}}}}
\def\DownBendAngle{12}
\def\UpBendAngle{45}
\def\SourceBendAngle{70}

\begin{tikzpicture}[>=Stealth, auto, node distance=2cm, line width=1.2pt]
\tikzset{
  L_edge/.style={->, very thick, blue!70!black},
  R_edge/.style={->, very thick, red!70!black},
  loop_large/.style={loop, in=135, out=45, looseness=20, distance=50pt}
}
\node[circle, inner sep=1pt, minimum size=4pt, draw=black, line width=0.5pt, fill=white] (N1) at (0.0, 24.0) {};
\node[circle, inner sep=1pt, minimum size=4pt, draw=black, line width=0.5pt, fill=white] (N2) at (-1.25, 21.0) {};
\node[circle, inner sep=1pt, minimum size=4pt, draw=black, line width=0.5pt, fill=white] (N3) at (-1.25, 15.0) {};
\node[circle, inner sep=1pt, minimum size=4pt, draw=black, line width=0.5pt, fill=white] (N4) at (-7.5, 6.0) {};
\node[circle, inner sep=1pt, minimum size=4pt, draw=black, line width=0.5pt, fill=white] (N5) at (-2.5, 3.0) {};
\node[circle, inner sep=1pt, minimum size=4pt, draw=black, line width=0.5pt, fill=white] (N6) at (0.0, 0.0) {};
\node[circle, inner sep=1pt, minimum size=4pt, draw=black, line width=0.5pt, fill=white] (N7) at (-1.25, 12.0) {};
\node[circle, inner sep=1pt, minimum size=4pt, draw=black, line width=0.5pt, fill=white] (N8) at (-6.25, 9.0) {};
\node[circle, inner sep=1pt, minimum size=4pt, draw=black, line width=0.5pt, fill=white] (N9) at (-5.0, 6.0) {};
\node[circle, inner sep=1pt, minimum size=4pt, draw=black, line width=0.5pt, fill=white] (N10) at (1.25, 9.0) {};
\node[circle, inner sep=1pt, minimum size=4pt, draw=black, line width=0.5pt, fill=white] (N11) at (1.25, 18.0) {};
\node[circle, inner sep=1pt, minimum size=4pt, draw=black, line width=0.5pt, fill=white] (N12) at (1.25, 12.0) {};
\node[circle, inner sep=1pt, minimum size=4pt, draw=black, line width=0.5pt, fill=white] (N13) at (6.25, 9.0) {};
\node[circle, inner sep=1pt, minimum size=4pt, draw=black, line width=0.5pt, fill=white] (N14) at (0.0, 6.0) {};
\node[circle, inner sep=1pt, minimum size=4pt, draw=black, line width=0.5pt, fill=white] (N15) at (0.0, 3.0) {};
\node[circle, inner sep=1pt, minimum size=4pt, draw=black, line width=0.5pt, fill=white] (N16) at (7.5, 6.0) {};
\node[circle, inner sep=1pt, minimum size=4pt, draw=black, line width=0.5pt, fill=white] (N17) at (-2.5, 6.0) {};
\node[circle, inner sep=1pt, minimum size=4pt, draw=black, line width=0.5pt, fill=white] (N18) at (3.75, 15.0) {};
\node[circle, inner sep=1pt, minimum size=4pt, draw=black, line width=0.5pt, fill=white] (N19) at (3.75, 12.0) {};
\node[circle, inner sep=1pt, minimum size=4pt, draw=black, line width=0.5pt, fill=white] (N20) at (3.75, 9.0) {};
\node[circle, inner sep=1pt, minimum size=4pt, draw=black, line width=0.5pt, fill=white] (N21) at (2.5, 6.0) {};
\node[circle, inner sep=1pt, minimum size=4pt, draw=black, line width=0.5pt, fill=white] (N22) at (2.5, 3.0) {};
\node[circle, inner sep=1pt, minimum size=4pt, draw=black, line width=0.5pt, fill=white] (N23) at (-3.75, 9.0) {};
\node[circle, inner sep=1pt, minimum size=4pt, draw=black, line width=0.5pt, fill=white] (N24) at (1.25, 15.0) {};
\node[circle, inner sep=1pt, minimum size=4pt, draw=black, line width=0.5pt, fill=white] (N25) at (-1.25, 18.0) {};
\node[circle, inner sep=1pt, minimum size=4pt, draw=black, line width=0.5pt, fill=white] (N26) at (1.25, 21.0) {};
\node[circle, inner sep=1pt, minimum size=4pt, draw=black, line width=0.5pt, fill=white] (N27) at (-3.75, 12.0) {};
\node[circle, inner sep=1pt, minimum size=4pt, draw=black, line width=0.5pt, fill=white] (N28) at (-1.25, 9.0) {};
\node[circle, inner sep=1pt, minimum size=4pt, draw=black, line width=0.5pt, fill=white] (N29) at (5.0, 6.0) {};
\node[circle, inner sep=1pt, minimum size=4pt, draw=black, line width=0.5pt, fill=white] (N30) at (-3.75, 15.0) {};
\draw[L_edge] (N1) to[bend right=\DownBendAngle] (N2);
\draw[L_edge] (N1) to[bend right=\DownBendAngle] (N25);
\draw[R_edge] (N1) to[bend left=\DownBendAngle] (N26);
\draw[R_edge] (N1) to[bend left=\DownBendAngle] (N30);
\draw[L_edge] (N2) to[bend right=\DownBendAngle] (N3);
\draw[R_edge] (N2) to[bend left=\DownBendAngle] (N11);
\draw[L_edge] (N3) to[bend right=\DownBendAngle] (N4);
\draw[R_edge] (N3) to[bend left=\DownBendAngle] (N7);
\draw[L_edge] (N4) to[bend right=\DownBendAngle] (N5);
\draw[L_edge] (N5) to[bend right=\DownBendAngle] (N6);
\draw[L_edge] (N7) to[bend right=\DownBendAngle] (N8);
\draw[L_edge] (N7) to[bend right=\DownBendAngle] (N9);
\draw[R_edge] (N7) to[bend left=\DownBendAngle] (N10);
\draw[R_edge] (N7) to[bend left=\DownBendAngle] (N5);
\draw[L_edge] (N8) to[bend right=\DownBendAngle] (N5);
\draw[R_edge] (N8) to[bend left=\DownBendAngle] (N9);
\draw[L_edge] (N9) to[bend right=\DownBendAngle] (N5);
\draw[R_edge] (N9) to[bend left=\DownBendAngle] (N5);
\draw[L_edge] (N10) to[bend right=\DownBendAngle] (N5);
\draw[R_edge] (N10) to[bend left=\DownBendAngle] (N5);
\draw[R_edge] (N10) to[bend left=\DownBendAngle] (N9);
\draw[L_edge] (N11) to[bend right=\DownBendAngle] (N12);
\draw[L_edge] (N11) to[bend right=\DownBendAngle] (N17);
\draw[R_edge] (N11) to[bend left=\DownBendAngle] (N18);
\draw[R_edge] (N11) to[bend left=\DownBendAngle] (N24);
\draw[L_edge] (N12) to[bend right=\DownBendAngle] (N5);
\draw[R_edge] (N12) to[bend left=\DownBendAngle] (N13);
\draw[L_edge] (N13) to[bend right=\DownBendAngle] (N14);
\draw[L_edge] (N13) to[bend right=\DownBendAngle] (N15);
\draw[R_edge] (N13) to[bend left=\DownBendAngle] (N16);
\draw[R_edge] (N13) to[bend left=\DownBendAngle] (N6);
\draw[L_edge] (N14) to[bend right=\DownBendAngle] (N6);
\draw[R_edge] (N14) to[bend left=\DownBendAngle] (N15);
\draw[L_edge] (N15) to[bend right=\DownBendAngle] (N6);
\draw[R_edge] (N15) to[bend left=\DownBendAngle] (N6);
\draw[L_edge] (N16) to[bend right=\DownBendAngle] (N6);
\draw[R_edge] (N16) to[bend left=\DownBendAngle] (N6);
\draw[R_edge] (N16) to[bend left=\DownBendAngle] (N15);
\draw[L_edge] (N17) to[bend right=\DownBendAngle] (N5);
\draw[R_edge] (N17) to[bend left=\DownBendAngle] (N15);
\draw[L_edge] (N18) to[bend right=\DownBendAngle] (N13);
\draw[R_edge] (N18) to[bend left=\DownBendAngle] (N19);
\draw[L_edge] (N19) to[bend right=\DownBendAngle] (N20);
\draw[L_edge] (N19) to[bend right=\DownBendAngle] (N22);
\draw[R_edge] (N19) to[bend left=\DownBendAngle] (N23);
\draw[R_edge] (N19) to[bend left=\DownBendAngle] (N21);
\draw[L_edge] (N20) to[bend right=\DownBendAngle] (N21);
\draw[L_edge] (N20) to[bend right=\DownBendAngle] (N22);
\draw[R_edge] (N20) to[bend left=\DownBendAngle] (N22);
\draw[L_edge] (N21) to[bend right=\DownBendAngle] (N22);
\draw[R_edge] (N21) to[bend left=\DownBendAngle] (N22);
\draw[R_edge] (N22) to[bend left=\DownBendAngle] (N6);
\draw[L_edge] (N23) to[bend right=\DownBendAngle] (N21);
\draw[R_edge] (N23) to[bend left=\DownBendAngle] (N22);
\draw[L_edge] (N24) to[bend right=\DownBendAngle] (N15);
\draw[R_edge] (N24) to[bend left=\DownBendAngle] (N19);
\draw[L_edge] (N25) to[bend right=\DownBendAngle] (N3);
\draw[R_edge] (N25) to[bend left=\DownBendAngle] (N6);
\draw[L_edge] (N26) to[bend right=\DownBendAngle] (N11);
\draw[R_edge] (N26) to[bend left=\DownBendAngle] (N27);
\draw[L_edge] (N27) to[bend right=\DownBendAngle] (N22);
\draw[R_edge] (N27) to[bend left=\DownBendAngle] (N28);
\draw[L_edge] (N28) to[bend right=\DownBendAngle] (N29);
\draw[R_edge] (N28) to[bend left=\DownBendAngle] (N29);
\draw[R_edge] (N29) to[bend left=\DownBendAngle] (N22);
\draw[L_edge] (N30) to[bend right=\DownBendAngle] (N6);
\draw[R_edge] (N30) to[bend left=\DownBendAngle] (N27);
\end{tikzpicture}}
\end{minipage}%
\begin{minipage}{0.32\textwidth}
  \centering
  \resizebox{!}{5cm}{\input{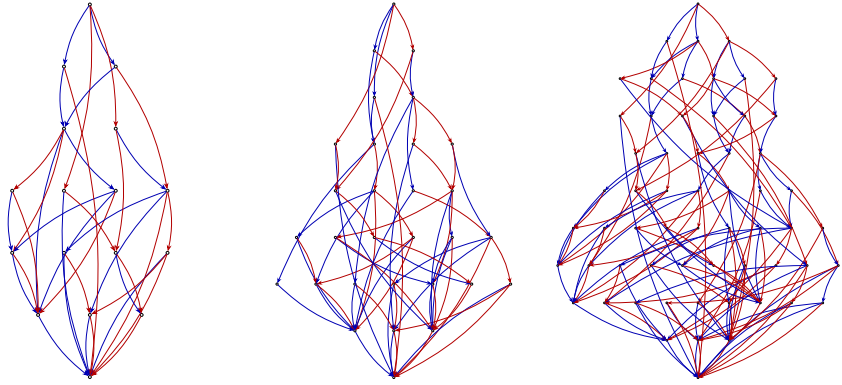}}
\end{minipage}%
\caption{The first three harrowing terms of $(H_n^1+H_n^2)$}
\label{fig:bach_sum}
\end{figure}

\subsection{Topological properties}

In this section we address some topological properties of $(\setC,wd)$.

% Unbounded
We first note that $(\setC,wd)$ is \emph{unbounded}, i.e. that the distance between elements of \setC\ can be arbitrarily high. We demonstrate this with two families, one partizan and one impartial. First, consider the sequence
 
$$ (s_n) = ( \{\pm1+n,\pm2+n,\ldots ,\pm(n-1)+n | \pm1-n,\pm2-n,\ldots ,\pm(n-1)-n\} )_{n=2}^\infty $$

This is a family of sums of switches that includes 

$$\pm\{3|1\}, \pm \{\{4|2\},\{5|1\}\}, \pm \{\{5|3\},\{6|2\},\{7|1\}\}, \ldots$$ 

The diminishing weight distance between successive terms increases without bound. As an impartial example consider the family of nimbers $\{*n\}_{n=1}^\infty$. For any $n>m>N\in \mathbb{N}$ we can see that the distance $wd(*n,*m)$ is bounded below by the difference in the degrees of the source vertices, $2(n-m)$.

% C is a separable metric space since it is countably infinite
A metric space is \emph{separable} if it contains a countable dense subset. That \setC\ has this property follows directly from the fact that \setC\ is itself countable.

% No limit points in $C$
% 	What is an open ball in (C,dwd)?
Since no two elements in \setC\ can have distance less than $\frac{1}{2^{d-1}}$, where $d$ is the larger distance from the source to the sink among the two games, it's clear that for each element of \setC\ there is an open ball containing \setC\ and nothing else. Hence \setC\ contains none of its limit points and is not \emph{complete} but is, in fact, a \emph{discrete space}. Since the collection of all such open sets, one for each element of \setC, comprises an open cover of \setC\ with no finite subcover, we see that $(\setC,wd)$ is not a \emph{compact space}. 

% Not proper - closed bounded sets are not compact, since not every open cover has a finite subcover.
% is this true? I don't understand proper as a property

% define discrete metric space - doesn't contain any of its limit points, which differs from a discrete metric

\section{Other distance metrics}\label{sec:othermetrics}

There are a number of alternative potential distance metrics on \setC, which we will devote this section to discussing. Firstly, let's consider some simpler metrics.

The simplest metric on any space is the \emph{discrete metric}, under which every pair of games has distance $1$ if they are distinct and distance $0$ if they are the same. There is not much else to say about this metric, so we  move on to some of potentially more interest.

% common birthday similarity
%	bs(G,H) = min birthday of game in \Delta{G.Followers,H.Followers}
% summed birthday distance = sum of (1/2)^bd over all games in \Delta
%	easy to compute in CGSuite since canonical form of a game is defined based on knowing all followers already
% common birthday distance = (1/2)^bd similarity
%	prove it's a metric
%	cauchy sequences - do they converge to transfinite games?
\begin{Definition}
Let $b(x)$ denote the birthday of the game $x$, which can be visualized as the length of a longest path from $x$ to $0$ in $D(x)$. Let $G,H\in \setC$. The following functions from $\setC \times \setC \rightarrow \mathbb{R}$ are distance metrics. 
\begin{enumerate}
	\item $bd(G,H) = \min \{b(x) | x\in \Delta(F(G),H(G))\}$, where $\Delta$ represents symmetric difference and $F(X)$ is the set of all followers of the game $X$
	\item $bs(G,H) = \sum_{x\in \Delta(G,H)} \left(\frac{1}{2}\right)^{b(x)}$
	\item $ed(G,H)$ is the traditional edit distance on the graphs $D(G),D(H)$
\end{enumerate}
\end{Definition}

% why is using an easy metric a bad idea for our purposes?
We omit proofs that the first two are metrics. Edit distance is a distance metric on the space of all finite graphs \cite{zhang1989simple} and can easily be extended to bicolored directed graphs by prohibiting color and direction changes. While these are relatively easy to compute, (or in the case of $ed$ already have existing algorithms leading to greater computational efficiency), none really meets the spirit of our goal to relate game distance to actual play. Diminishing weight distance, on the other hand, gives appropriate priority to early moves over later moves, and better compares board states. Also, of all of these only $bs$ has the potential for interesting converging sequences.

On the matter of board states, Carvalho and dos Santos show in \cite{carvalho2019nontrivial} that every game in \setC\ can be realized by a position in the universal game \textsc{K}$\bar{\textsc{o}}$\textsc{nane} given a large enough board. Thus a simple to describe (but very difficult to compute) metric on \setC\ could be designed around the difference between these board states in this specific ruleset. However, the associated boards in \textsc{K}$\bar{\textsc{o}}$\textsc{nane} become unmanageably large quite fast, and we also do not wish to bias our distance metric towards this ruleset.

Another alternative distance metric is based on an idea of Aaron Siegel's \cite{SiegelPersonal:2024} addressing similarity among a pair of short games. Given $G\in \setC$ let $G_{\triangleleft}^n$ (alternately $G_{\triangleright}^n$) be those games born by day $n$ that are less than (alternately greater than) or confused with $G$. Then let $G^n$ be the game $\{G_{\triangleleft}^n|G_{\triangleright}^n\}$. Siegel suggests the similarity between games $G$ and $H$ can be measured by the greatest $n$ such $G^n = H^n$. We define the \emph{resolution distance} between $G,H\in \setC$ to be $rd(G,H) = \left( \frac{1}{2}\right)^{d}$ where $d = \min_{i=0}^\infty \{i | G^i \neq H^i\}$ if $G\neq H, rd(G,G)=0$. Note that if two games $G,H$ are in different outcome classes then $G^0 \neq H^0$ so their resolution distance is $1$. 

Resolution distance is interesting and worth further study, however it yields only values that are powers of $\frac{1}{2}$ and, most importantly, is incredibly difficult to calculate. As of this writing we do not even know all games born by day $4$, which is somewhere between $10^{28.2}$ and $4\times 10^{184}$  \cite{suetsugu2022improving}. Analysis including investigation of topological properties and discovery of Cauchy sequences might prove interesting, but computing $rd(G,H)$ appears to be intractable.

%\section{Implementation}\label{sec:implementation}

% what is a genetic algorithm? simulated annealing?
% Given a game value and a ruleset, can we find a game position with that value?
%	searching for unknown values - +-2.5 or *4 in domineering, *4 in clobber
%	for demonstrations in teaching environments

% limitations of CGSuite
%	p-rng (Random has been implemented as of 2.2 beta)
%		random := Random(1474)
%		random.NextInteger(100) for integer  0 <= x < 100
%	data structures
%	file i/o
% advantages of CGSuite
%	fast game arithmetic computation
%	automated top-down memoization of game values

% results of implementing tractable distance metrics and approximations of other metrics

\section{Conclusion and remaining questions}\label{sec:conclusion}

We have motivated the case for, and developed, a real-valued distance metric on the set of all short combinatorial games in canonical form, and have provided a number of loopy games in the closure of the resulting metric space. In doing so we have investigated sidling sequences, as well as sequences of finite games that do not sidle.

% What is the completion of (C,dwd)?
%	Does it include all plumtrees? All stoppers? 
%	We've seen some stoppers and some non-stoppers, but some non-stoppers aren't in \overline{C} like tis and tisn and dud, and even non stopper sided games like BC aren't in \overline{C}. Are *any* non stopper-sided games in \overline{C}? like bach1+bach2?
Loopy games are partitioned into stoppers and non-stoppers, enders and non-enders, and even stopper-sided and non-stopper-sided. We have now seen that the distance metric $wd$ distinguishes loopy games yet again, those in $\overline{\setC}$ and those outside of it. We have not seen a plum tree that is not in $\overline{\setC}$, nor a non-stopper-sided game in $\overline{\setC}$. However, we have seen some non-stoppers in $\overline{\setC}$ and others that are not. The following questions remain.

\begin{enumerate}
	\item Are all plum trees in $\overline{\setC}$?
	\item Are any non-stopper-sided games in $\overline{\setC}$?
	\item What distinguishes non-stoppers in $\overline{\setC}$ from those not in $\overline{\setC}$?
\end{enumerate}

% Can we consider the completion of \overline{C} by extending dwd to allow for sequences of loopy games?
We have also not addressed Cauchy sequences composed entirely of infinite DAGs. Could sequences of this form help us expand $\overline{\setC}$?

% Nimbers are still elusive in practical tests of the genetic algorithm; how do we get there?

\bibliographystyle{plainurl}

\end{document}